\documentclass[a4paper,oneside,11pt]{article} 
\usepackage[utf8]{inputenc} 
\usepackage[T1]{fontenc} 
\usepackage{textcomp} 
\usepackage[english]{babel} 
\usepackage{indentfirst}
\usepackage[autolanguage]{numprint} 
\usepackage{hyperref} 
\hypersetup{colorlinks=true,linkcolor=blue,citecolor=red,urlcolor=blue}
\usepackage{graphicx} 
\usepackage{verbatim} 
\usepackage{makeidx} 
\usepackage[refpage]{nomencl} 
\usepackage{enumitem} 
\usepackage{tikz} 
\usetikzlibrary{matrix,arrows,patterns}
\usepackage{multicol}
\usepackage{fullpage}
\usepackage{comment}
\usepackage{stmaryrd}
\usepackage{ulem}
\usepackage{tikz}
\usepackage[mathscr]{euscript}

\usepackage{amsmath,amssymb,amsthm,amsfonts} 
\usepackage[all]{xy} 

\numberwithin{equation}{section}

\newcounter{Main}

\theoremstyle{plain} 

\swapnumbers 
\theoremstyle{definition} 
\newtheorem{Def}{Definition}[section] 
\newtheorem{Def,Thm}[Def]{Definition and theorem} 
\newtheorem{Def,Prop}[Def]{Proposition-definition} 

\theoremstyle{plain} 
\newtheorem{Prop}[Def]{Proposition} 
\newtheorem{Lemma}[Def]{Lemma} 

\newtheorem{Theorem}[Def]{Theorem} 
\newtheorem{Corollary}{Corollary}[Def] 
\theoremstyle{remark} 
\newtheorem{Example}[Def]{Example} 
\newtheorem{Remark}[Def]{Remark} 

\usepackage{xcolor}
\usepackage[all]{xy}

\newcommand\sbmattrix[4]{\textnormal{\scriptsize$\left(\begin{array}{cc}#1&#2\\#3&#4\end{array}\right)$\normalsize}}

\usepackage{mathtools}

\title{Computing quaternionic representations via twisted forms of Bruhat-Tits trees} 
\author{Luis Arenas-Carmona\footnote{Universidad de Chile, Facultad de Ciencias, Casilla 652, Santiago, Chile. Email: learenas@u.uchile.cl.} \, and Claudio Bravo\footnote{Universidad de Talca, Instituto de Matemáticas, Talca, Chile. Email: claudio.bravo@utalca.cl}} 

\date{}

\begin{document} 

\maketitle

\begin{abstract}
This work is devoted to the study 
of representations
of finite subgroups of the group 
of units of quaternion division 
algebras over a global or local 
field arising from the inclusion via extension 
of scalars splitting the algebra.
Following a question by Serre, we 
study the set $\mathrm{IF}$ of conjugacy 
classes of integral representations 
that are conjugates of the given representation 
over the field. The set $\mathrm{IF}$ is often called the set of integral forms in the literature.
In previous works, we have seen that, for a given
representation, the set $\mathrm{IF}$ can be indexed by the vertex
set of a suitable subgraph of the Bruhat-Tits tree for
the special linear group. In this work, we describe 
a construction that allows the simultaneous study
of the set $\mathrm{IF}$ over different splitting fields.
For this, we devise and use a theory of twisted Galois 
form of Bruhat-Tits trees. With this tool,
we explicitly compute, in most cases,
the cardinality of $\mathrm{IF}$ for the representation of the classical
quaternion group of order $8$ studied by Serre, Feit and 
others, as much as for other similar groups.
\end{abstract}


\section{Introduction}\label{section intro}

Representation theory is the area of mathematics devoted to understanding how algebraic structures can be described as linear transformations of vector spaces and related structures.
This domain has a significant connection with other classical areas such as physics, geometry, algebra and number theory.

In all of this work, upper case letters like $K$ denote 
local fields. Lower case letters like $k$ denote  global
fields (often number fields). Finally, we use $\mathtt{K}$ 
for a field that could be either,
local or global. In all cases, the symbols
$\mathcal{O}_K$, $\mathcal{O}_k$ or
$\mathcal{O}_\mathtt{K}$ denote the corresponding 
ring of integers. For global function fields
this requires the choice of a maximal affine subset,
but we never explicitly refer to this case,
so this might be ignored.

For a global or local field $\mathtt{K}$, the (integral) 
$\mathcal{O}_\mathtt{K}$-representations 
of a group $G$, i.e., the homomorphisms 
$G \to \mathrm{GL}_n(\mathcal{O}_\mathtt{K})$, play a significant 
role in various arithmetic or geometrical phenomena, including the 
arithmetic of group rings and orders, the theory of integral lattices 
and genera, local–global structures in arithmetic geometry, and the 
study of integral invariants arising in algebraic K-theory, such as 
Grothendieck and Whitehead groups (see \cite{Reiner}).
In analogy to the
representation theory on $\mathtt{K}$, one possible approach 
to study these 
representations is to classify $\mathcal{O}_\mathtt{K}[G]$-modules.
Using that approach, integral representations of cyclic and 
related groups have already been studied in works such as
\cite{HellerReiner1,HellerReinerII,Reiner1,ReinerRoggenkamp,Kang}.

Let us fix a linear representation 
$\rho: G \to \mathrm{GL}_n(\mathtt{K})$ 
of a finite group $G$. 
We denote by $\mathrm{IF}_{\rho}^{G}(\mathtt{K})$ the set of 
integral forms of $\rho$, i.e.,
 $\mathrm{GL}_n(\mathcal{O}_\mathtt{K})$-conjugacy 
classes of integral representations that are 
$\mathrm{GL}_n(\mathtt{K})$-conjugates of
$\rho$. In 1997,  J.-P. Serre asks, in $3$ letters to 
W. Feit, whether it is possible that 
$\mathrm{IF}_{\rho}^{G}(\mathtt{K})=\varnothing$.
He considers a number field $\mathtt{K}=k$. Then he proceeds to
answer this question for the particular case where $G$ is the 
quaternion group $Q_8$, $\rho$ is its unique
faithful two dimensional representation,
and $k$ is a quadratic field over which $\rho$ is defined.
The fact that the Schur index of $\rho$ is $2$, and therefore
$\rho$ fails to have a unique minimal field of definition,
makes this question significantly harder to answer.

Assume that $\mathrm{char}(\mathtt{K})\neq 2$, 
and let $\mathfrak{A}_\mathtt{K}= 
\big(\frac{a,b}{\mathtt{K}}\big)$ be a 
quaternion division algebra.
A quaternionic group $G$ is a finite group 
$G\subseteq\mathfrak{A}_\mathtt{K}^*$ of invertible quaternions.
The usual quaternionic group $Q_8$, for instance, is defined 
by taking $\mathfrak{A}_{\mathbb{Q}}=
\big(\frac{-1,-1}{\mathbb{Q}}\big)$.
These groups have, by definition, a faithful two dimensional
representation of Schur index $2$, whose study shares
the difficulties mentioned above.
A splitting field for $\mathfrak{A}_\mathtt{K}$ is any extension
$\mathtt{L}/\mathtt{K}$ for which $\mathfrak{A}_\mathtt{L}:=
 \mathtt{L}\otimes_\mathtt{K}\mathfrak{A}_\mathtt{K}$
 is isomorphic to a matrix algebra. Quaternionic groups
 have a natural representation for every splitting field.
 It is well known that,
 when $\mathtt{K}=K$ is a local field, then every even
 degree finite extension is a splitting field, as follows from 
 the behavior of the Hasse invariant under extensions, see for 
 instance \cite[\S XIII, Cor. 3]{Serre-localfields}. When 
 $\mathtt{K}=k$ is a global field, an algebraic extension is a 
 splitting field if it is so for every completion, by the 
 Albert–Brauer–Hasse–Noether Theorem.

In \cite{CliffRitterWeiss} Cliff, Ritter and Weiss had shown that the irreducible $2$-dimensional (complex) representation $\bar{\rho}$ of the quaternion group $G=Q_8$ can be realized over $k=\mathbb{Q}(\sqrt{-35})$ but not over $\mathcal{O}_k$.
More generally, consider an imaginary quadratic field 
$k=\mathbb{Q}(\sqrt{-N})$ where $N$ is a square-free integer 
such that $N\equiv 3(\text{mod } 8)$, so that $\bar{\rho}$ 
can be realized over $k$. It is shown in \cite{SerreLocalRep} 
that $\bar{\rho}$ can be realized over 
$\mathcal{O}_k$ if and only if $N$ can be written in 
the form $x^2+2y^2$ for some pair of integers $x,y$,
i.e., whenever every prime divisor of $N$ is congruent 
$(\text{mod } 8)$ to either $1$ or $3$.
Moreover, the article explains how this result can be derived 
from the theory of genera, as developed by Gauss in his 
Disquisitiones and by Hilbert in his Zahlbericht.

When $\mathrm{IF}_{\rho}^{G}(\mathtt{K})$ is nonempty, one can ask for its cardinality.
In this work, we answer this latter question for some well-known 
quaternionic groups, as for instance, for $Q_8$, the Hurwitz 
group $\mathrm{SL}_2(\mathbb{F}_3)$ and the dicyclic group, over
arbitrary local fields $K$ whose characteristic is not $2$. Note
that, for a local field $K$, the set 
$\mathrm{IF}_{\rho}^{G}(K)$ of local 
integral forms is always 
nonempty. This can be easily proven from Prop. 2 in 
\cite[\S II.1.3]{SerreTrees}.  In fact, a finite group
is bounded and the determinants of its elements are
units, so it is always contained in a 
maximal compact subgroup, which is a conjugate of $\mathrm{GL}_2(\mathcal{O}_K)$.

Let $K$ be a local field, let $G$ be a finite group, and let 
$\rho$ be a $n$-dimensional representation $\rho:G \to \mathrm{GL}_n(K)$.
In \S \ref{branches and Int rep}, we 
describe an explicit bijection between
the set $\mathrm{IF}_{\rho}^{G}(K)$ and the vertex 
set of certain subcomplex of the Bruhat-Tits building 
$\mathcal{X}=\mathcal{X}(\mathrm{SL}_n,K)$. 
This complex, called the branch of $G$ in the sequel, is defined 
as the largest subcomplex of $\mathcal{X}$ 
whose vertices correspond to maximal orders in 
$\mathbb{M}_n(K)$ 
containing the image $\rho(G)$.
For $n=2$, branches 
have already been described by the authors in
\cite{ArenasBranches,ArenasSaavedra,ArenasBravoExt,ArenasBravoCar2}.
For general $n$,
we still have some partial descriptions 
due to Shemanske, Linowitz, El Maazouz, Nebe, Stanojkovski 
and the first author. See
\cite{Shemanske,ArenasBordeaux,LinShe,Nebe}.

The $\mathcal{O}_K$-module $\mathcal{O}_K[\rho(G)]$ 
generated by $\rho(G)$ 
is an $\mathcal{O}_K$-order, i.e., a $\mathcal{O}_K$-
lattice with a ring 
structure, as a subset of $\mathbb{M}_n(K)$.
Assuming that $\mathcal{O}_K[\rho(G)]$ is maximal, 
the integral representations of $G$ have been studied 
by Manilin in \cite{Manilin1,Manilin2,Manilin3},
among other works.
However, this is not often the case. For instance, if $G$
is a quaternionic group, then it is contained in the unique maximal order of the corresponding quaternion algebra
$\mathfrak{A}_\mathtt{K}$, and its image in 
$\mathfrak{A}_\mathtt{L}$,
for a splitting field $\mathtt{L}$, is never maximal
(see \S \ref{subsection int rep max orders}).
In the non-maximal case, integral representations of some 
finite subgroups of $\mathrm{GL}_2(K)$, 
such as cyclic and dihedral groups, have been studied in 
\cite{ArenasAguiloSaavedra} by using branches 
in Bruhat-Tits trees. This theory does apply to
quaternionic groups, but it fails to allow the simultaneous 
study of multiple splitting fields.

In this work, we develop a method that extends the
theory in \cite{ArenasAguiloSaavedra}, while allowing
the study of representations
of quaternionic groups that arise from extending
scalars to multiple splitting fields.
To do so, we introduce in \S 
\ref{subsection forms of BTT}-\S 
\ref{section branches in galois forms} 
a theory of Galois twisted forms of Bruhat-Tits trees and branches. Our main result in this direction is Th. 
\ref{main teo 0}. In \S \ref{subsection int rep max orders} we 
describe in this terms the local integral representations,
over a splitting field $L$ of the maximal order in a quaternion
division $K$-algebra $\mathfrak{A}_K$. This is the content
of Cor. \ref{main teo 1}. Descriptions of the 
$\mathcal{O}_L$-order spanned by the maximal 
$\mathcal{O}_K$-order in $\mathfrak{A}_K$ are known,
see for instance \cite[Lemmas 2.10-11]{LiXueYu}.
However, our method gives us some additional information.
For instance, over a field containing both, a ramified
and an unramified quadratic extension, the representations
defined over either can be explicitly identified.
Then we apply this result to the description 
of local integral representations of some classical 
quaternionic  groups such as $Q_8$,
$\mathrm{SL}_2(\mathbb{F}_3)$
or the dicyclic group $C_3 \rtimes C_4$. See \S 
\ref{subsection int rep max orders} and 
Th. \ref{main teo 2}-\ref{main teo 4}.
We finally illustrate how this theory can be applied
to the study of global representations.
In particular, Th. \ref{main teo 5} describes the number of 
conjugacy classes of global integral representations of $Q_8$ 
in the case where 
$\mathrm{IF}_{\rho}^{G}(k)\neq\varnothing$.
Throughout, we use the notation $\sharp S$ for
the cardinality of a set $S$.

\subparagraph{Convention on graphs.}\label{subsection conventions}
In the sequel, by a combinatorial graph $\mathfrak{c}$,
we mean a set $V$, whose elements are called vertices,
together with a set $A$ of pairs of elements in
$V$ called edges. Every combinatorial graph 
has an underlying topological space 
$C$ associated to $\mathfrak{c}$. 
The space $C$ is obtained from $\mathfrak{c}$ 
by gluing a copy of the interval 
$[0,1]\subseteq\mathbb{R}$ by the 
endpoints to the vertices $v$ and $w$, for every edge
$\{v,w\}\in A$. The vertex set $V$ is, therefore,
a discrete subset of $C$. The pair
$\mathcal{C}=(\mathfrak{c},C)$
is called the topological graph associated to
$\mathfrak{c}$. We call $\mathcal{C}$ a topological tree
when $\mathfrak{c}$ is a tree.
A continuous map $f:\mathcal{C}\rightarrow\mathcal{C}'$
between topological graphs is just a map
between the corresponding topological spaces.
However, such $f$ is called a simplicial map
if it is induced from a suitable morphism of
combinatorial graphs, i.e., a map for the vertex set $V$
of $\mathcal{C}$ to the vertex set $V'$
of $\mathcal{C}'$ sending edges to edges.
We write $x\in\mathcal{C}$ instead of $x\in C$. 
We say that a subset $S$ of the space $C$ is a
subgraph, when it is the subspace corresponding to a subgraph of
$\mathfrak{c}$. In this case, we often denote the topological subgraph 
as $\mathcal{S}$ and use expressions like \textit{let
$\mathcal{S}\subseteq \mathcal{C}$ be a subgraph}.
Similarly, the expression $\mathcal{C}\subseteq\mathcal{C}'$
means that the space $C$ is contained into the space $C'$
corresponding to $\mathcal{C}'$, usually after a suitable 
identification. In this case, we do not assume that
$\mathcal{C}$ is a subgraph unless explicitly stated.
For instance, the set of vertices in $\mathcal{C}$ might not coincide
with the set of vertices of $\mathcal{C}'$ that are contained in 
$\mathcal{C}$. Other similar conventions are used. 
When we mean that $x$ is a vertex of the graph $\mathfrak{c}$
we write $x\in V(\mathcal{C}):=V(\mathfrak{c})$ instead.

\begin{figure}
    \unitlength 1mm 
\linethickness{0.4pt}
\ifx\plotpoint\undefined\newsavebox{\plotpoint}\fi 
    \[
    \begin{picture}(30,14)(0,40)
\put(15.2,43.2){$\bullet$}
\put(15.2,55.4){$\bullet$}
\put(16,44){\line(0,1){12}}
\put(3.2,43.2){$\bullet$}
\put(3.2,55.4){$\bullet$}
\put(4,56){\line(1,0){12}}
\put(4,44){\line(1,0){24}}
\put(27.2,43.2){$\bullet$}
\end{picture}\qquad\qquad 
\begin{picture}(30,14)(0,40)
\put(15.2,43.2){$\bullet$}
\put(15.2,55.4){$\bullet$}
\put(16,44){\line(0,1){12}}
\put(15.2,47.2){$\bullet$}
\put(15.2,51.4){$\bullet$}
\put(3.2,43.2){$\bullet$}
\put(3.2,55.4){$\bullet$}
\put(4,56){\line(1,0){12}}
\put(7.2,55.2){$\bullet$}
\put(11.2,55.2){$\bullet$}
\put(4,44){\line(1,0){24}}
\put(7.2,43.2){$\bullet$}
\put(11.2,43.4){$\bullet$}
\put(19.2,43.2){$\bullet$}
\put(23.2,43.4){$\bullet$}
\put(27.2,43.2){$\bullet$}
\end{picture}
\]
    \caption{A graph $\mathcal{C}$ (left) and its 
    subdivision $\mathcal{C}(3)$ (right).}
    \label{f1n}
\end{figure}
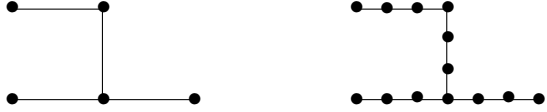

If $\mathcal{C}=(\mathfrak{c},C)$ is a topological tree, 
the space $C$ has a natural metric $d$ that reduces to the usual
metric on each interval and 
satisfies $d(x,z)=d(x,y)+d(y,z)$ whenever $x$ and $z$ are in 
different connected components of $C \smallsetminus \{y\}$.
The $n$-th subdivision $\mathcal{C}(n)$ of a topological
graph $\mathcal{C}$ is the graph obtained by replacing
every edge in $\mathcal{C}$ by a path of length
$n$, as illustrated in Fig. \ref{f1n} for $n=3$.
Note that the topological spaces subjacent to
$\mathcal{C}(n)$ and $\mathcal{C}$ are homeomorphic.
The metric in $\mathcal{C}(n)$ is often normalized, 
in what follows, so that these two spaces can be
identified.

The real line can be turn into a topological graph
$\mathcal{R}=(\mathfrak{r},\mathbb{R})$ in a way that every
integer is a vertex. Integral intervals are connected
(topological) subgraphs of $\mathcal{R}$. A line $\mathcal{P}$ in a 
topological tree $\mathcal{C}$ is the image (as a topological
graph) of  an injective simplicial map defined on 
an integral interval $\mathcal{J}$. When
$\mathcal{J}=\big(\mathfrak{j},[0,\infty)\big)$, for
a suitable graph $\mathfrak{j}$, then
$\mathcal{P}$ is called a ray. If $\mathcal{J}=\mathcal{R}$,
then $\mathcal{P}$ is called a maximal line. Two rays are
called equivalent if their intersection is a ray. Equivalent
classes of rays are called visual limits and often drawn as
stars in diagrams. The set of visual limits in a graph
$\mathcal{C}$ is denoted by $\partial_\infty(\mathcal{C})$.
A ray in a tree has a unique visual limit
and a maximal line has two. A simplicial embedding
$f:\mathcal{C}\rightarrow\mathcal{C}'$, between trees,
induces an inclusion $\partial_\infty(f):
\partial_\infty(\mathcal{C})\rightarrow
\partial_\infty(\mathcal{C}')$.
The visual limits of the subdivision
$\mathcal{C}(n)$ can be identified naturally with those of
$\mathcal{C}$.

\section{Main results}\label{section main results}

In the sequel, we consider a quaternionic group $G$ and the
representation $\rho$ obtained by extension of scalars to an
arbitrary field extension $\mathtt{L}/\mathtt{K}$ splitting the
algebra.  In this context, we describe the set 
$\mathrm{IF}_{\rho}^{G}(\mathtt{L})$ of 
$\mathcal{O}_\mathtt{L}$-conjugacy classes of 
integral representations
that are $\mathtt{L}$-conjugates of $\rho$, 
assuming that it is not empty.
The conditions under which 
$\mathrm{IF}_{\rho}^{G}(\mathtt{L})\neq\varnothing$ are
already known for the examples treated here.
This was discussed in \S\ref{section intro} for the local
case, which is the one we consider until Th. \ref{main teo 4}. 
See \S3 of the last letter in \cite{SerreLocalRep} for the 
particular global
cases considered in Th. \ref{main teo 5}. 

Here, we consider the local case $\mathtt{K}=K$.
Note that, when the quaternionic group is contained in 
a split algebra $\mathfrak{A}$ over $K$, 
i.e., when  $\mathfrak{A} \cong \mathbb{M}_2(K)$, we 
can naturally define a tree from $\mathfrak{A}$ by 
considering the maximal $\mathcal{O}_K$-orders in
$\mathfrak{A}$ as vertices, as described in 
\cite{ArenasBranches}. This tree coincides with the 
Bruhat-Tits tree of $\mathrm{SL}_2(K)$.
In particular, we can apply the machinery developed in 
\S \ref{branches and Int rep} in order to determine the 
integral representation of finite multiplicative
groups contained in split quaternion algebras.
In order to study the behavior of maximal 
orders in the matrix algebra under a field extension
it is useful to define the tree $\mathcal{T}$ in a way
that allows us to consider the tree $\mathcal{T}_K$,
whose vertices are the maximal $\mathcal{O}_K$-orders
in $\mathbb{M}_2(K)$, as a subset of the tree 
$\mathcal{T}_L$, whose vertices are the maximal 
$\mathcal{O}_L$-orders in $\mathbb{M}_2(L)$.
The easiest way to do this is defining $\mathcal{T}$
as a topological graph as described at the 
end of \S\ref{section intro}.
The embedding $\mathcal{T}_K\hookrightarrow\mathcal{T}_L$
is an embedding of topological spaces that is induced
by a morphism of graphs only when $L/K$ is an unramified
extension. For ramified extensions, neighboring vertices
of $\mathcal{T}_K$ are no longer neighbors as vertices
in $\mathcal{T}_L$. They lie at a distance equal
to the ramification index $e=e(L/K)$.
In particular, the subdivision $\mathcal{T}_K(e)$
is a subgraph of $\mathcal{T}_L$. 
See \S \ref{subsection ext BTT} for
additional details.

In contrast, when $\mathfrak{A}$ is a division algebra, 
i.e., when it fails to be split, we can not directly 
define a tree from $\mathfrak{A}$.
In this case, we have to take an algebraic (Galois) 
extension $L/K$ splitting $\mathfrak{A}$ and then to 
consider the tree $\mathcal{T}=
\mathcal{X}(\mathrm{SL}_2,L)$ defined 
from $\mathfrak{A}_L=\mathfrak{A} \otimes_K L \cong 
\mathbb{M}_2(L)$. Note that $\mathfrak{A}$ can be view 
as a $K$-subalgebra of $\mathfrak{A}_L$, as we do in the 
sequel. The main issue 
with this technique is that we fail to have a unique 
minimal splitting field, so we are bound to make a 
different computation for every such field.
As we prove in \S \ref{subsection forms of BTT}, 
this problem can be overcome by defining the tree as a twisted
$\mathcal{G}$-form $\widehat{\mathcal{T}}$
of the usual Bruhat-Tits tree 
$\mathcal{T}$, where $\mathcal{G}=\mathrm{Gal}(L/K)$.
In the cohomological language, recall that the pointed set
$H^1\big(\mathcal{G}, \mathrm{PGL}_2(L)\big)$
 indexes all the quaternion algebras splitting at $L$.
 Similarly, if we denote by $\mathrm{Simp}(\mathcal{T})$ the 
 $\mathcal{G}$-group of all simplicial homeomorphisms of 
 $\mathcal{T}=\mathcal{T}_L$, then the pointed set  
 $H^1\big(\mathcal{G}, \mathrm{Simp}
 (\mathcal{T})\big)$ indexes the isomorphism classes
of twisted $\mathcal{G}$-forms of $\mathcal{T}$, as 
$\mathcal{G}$-sets, in the sense described in
\cite[\S III.1]{Serre-CohomologieGaloisienne}. 
See Prop. \ref{prop cohgal for trees} for details.
These forms inherit the
structure as topological graphs from $\mathcal{T}$, 
since $\mathcal{G}$ acts by simplicial maps.
The tree associated to
$\mathfrak{A}$ can be alternatively defined as a 
representative of the class of the $\mathcal{G}$-forms of 
$\mathcal{T}$ defined by $[\mathfrak{A}]$ via the natural 
map:
$$H^1\big(\mathcal{G}, \mathrm{PGL}_2(L)\big) \to 
H^1\big(\mathcal{G}, \mathrm{Simp}(\mathcal{T})\big).$$
See Lemma \ref{lemma injective arrow for forms} for 
details. The tree $\widehat{\mathcal{T}}$ defined
in this way is expected to satisfy the same conventions as
the classical Bruhat-Tits tree.
Indeed, in \S\ref{section branches in galois forms}, 
we prove next result:

\begin{Theorem}\label{main teo 0}
Let $\mathfrak{A}$ be a quaternion division algebra 
over $K$. Let $\widehat{\mathcal{T}}$ be the twisted
form corresponding to $\mathfrak{A}$.
Let $\mathfrak{H}$ be an order in 
$\mathfrak{A}$, and write
$\mathfrak{H}_E\subseteq\mathfrak{A}_E$ for the 
$\mathcal{O}_E$-order it generates, for every 
intermediate field $K \subseteq E \subseteq L$. Then the following 
statements hold:
\begin{itemize}
\item[(1)] There is an inclusion that preserves the correspondence 
$E\mapsto\widehat{\mathcal{T}}_E$, which associates a subspace
of $\widehat{\mathcal{T}}$, homeomorphic to the Bruhat-Tits 
tree of $E$, to every field $E$ with 
$K\subseteq E\subseteq L$ that splits $\mathfrak{A}$.
Furthermore, $\widehat{\mathcal{T}}_E$ is
a topological graph, and the embedding
$\widehat{\mathcal{T}}_E\big(e(L/E)\big)
\hookrightarrow\widehat{\mathcal{T}}$ is simplicial.
\item[(2)] There exists an explicit subtree $\mathcal{S}$ of
$\widehat{\mathcal{T}}$, independent of $E$, such that the
maximal orders in $\mathfrak{A}_E$ containing 
$\mathfrak{H}_E$ are naturally in correspondence
with the vertices of the tree 
$\widehat{\mathcal{T}}_E$
lying inside $\mathcal{S}$, for every intermediate
field $E$ splitting $\mathfrak{A}$. Furthermore, the subset
$\mathcal{S}_E:= 
\mathcal{S}\cap\widehat{\mathcal{T}}_E$
is a subgraph of $\widehat{\mathcal{T}}_E$.
\end{itemize}
\end{Theorem}

Next result, which is proved in \S 
\ref{subsection int rep max orders}, is a straightforward 
consequence of Th. \ref{main teo 0}.

\begin{Corollary}\label{main teo 1}
Let $\mathfrak{A}$ be a quaternion division algebra over $K$, 
and let $\mathfrak{D}$ be its unique maximal order. 
Let $\Delta \in \mathcal{O}_K^*$ be an unramified unit 
in $\mathcal{O}_K$.
Let $E_1/K$ be an algebraic extension containing $\sqrt{\Delta}$,
while $E_2/K$ is an algebraic extension containing
a quadratic extension of $K$, but not $\sqrt{\Delta}$.
Denote by $e_i=e(E_i/K)$, for $i=1,2$, the corresponding 
ramification index. Then, the following statements hold:
\begin{itemize}
\item[(1)] All representations of the form 
$\rho:\mathfrak{D} \to \mathbb{M}_2(\mathcal{O}_{E_2})$ 
are conjugates.
\item[(2)] There are exactly $e_1+1$ conjugacy classes of representations $\psi:\mathfrak{D} \to \mathbb{M}_2(\mathcal{O}_{E_1})$.
\item[(3)] The representation $\rho$ in (1) is defined
over $E_1$ precisely when $e_1$ is even.
\end{itemize}
\end{Corollary}

Let $G$ and $\rho$ be as introduced at the beginning
of this section.
Consider the order $\mathfrak{H}=\mathcal{O}_K[G]$ 
generated by $G$. Integral representations of $G$ over 
every intermediate field $E$, as above, can be described 
in terms of the vertex set $V(\widehat{\mathcal{T}}_E)$ of the 
twisted tree $\widehat{\mathcal{T}}_E$ thanks to 
Th. \ref{main teo 0}.
Indeed,
Th. \ref{main teo 0b} below 
generalizes previously known results that
we recall in \S \ref{branches and Int rep}.
See Prop. \ref{prop equiv rep and conj class}.

\begin{Theorem}\label{main teo 0b}
Let $\mathfrak{L}$ be the centralizer of $\rho$.
Then, the natural action of $\mathfrak{L}_L^*$,
by conjugation on maximal orders,  
induces an
action on the space $\mathcal{S}$, defined in Th. \ref{main teo 0},
that is compatible with both, the Galois action
and the graph structure. In particular, it
induces a simplicial action of $\mathfrak{L}_E^*$ on 
$\mathcal{S}_E$, for every intermediate field
$E$ splitting the algebra.
Furthermore, there exists a bijection between $\mathrm{IF}_{\rho}^{G}(E)$
and the set of $\mathfrak{L}_E^*$-orbits 
of vertices in $\mathcal{S}_E$.
In particular, if $\rho$ is absolutely irreducible, then $\mathrm{IF}_{\rho}^{G}(E)$ is in bijection with the vertices of $\widehat{\mathcal{T}}_E$ corresponding to maximal orders 
containing $\rho(G)$.
\end{Theorem}

By using Th. \ref{main teo 0}, in \S \ref{subsection int rep max orders}  we 
describe the integral representations of some quaternion groups 
$G$ spanning a maximal order. These representations are
certainly absolutely irreducible since their images span the
algebra.
This result certainly describes 
the integral local representations of the Hurwitz unit group
 $G_H=\langle \mathtt{w}, \mathtt{u},\mathtt{v} 
\rangle\subseteq \mathfrak{A}^*$, where 
$$\mathfrak{A}=\left(\frac{-1,-1}{\mathbb{Q}}\right)=
\mathbb{Q}\big[ \mathtt{u},\mathtt{v} 
\vert \mathtt{u}^{2}=\mathtt{v}^{2}=-\mathtt{1}, 
\mathtt{u}\mathtt{v}+\mathtt{v}\mathtt{u}=\mathtt{0} \big]
\textnormal{ and } 
\mathtt{w}=\frac{1}{2}(-\mathtt{1}+
\mathtt{u}+\mathtt{v}+\mathtt{u}\mathtt{v}) \in \mathfrak{A}.$$
The group $G_H$ is isomorphic to $\mathrm{SL}_2(\mathbb{F}_3)$,
according to \cite[Lemma 11.2.1]{Voight}.
In \S \ref{subsection int rep max orders} we prove next 
result, which describes the integral $p$-adic representations of 
the Hurwitz unit group:

\begin{Theorem}\label{main teo 2}
Let $E$ be a Galois extension of $\mathbb{Q}_p$ containing a 
quadratic extension and $\rho: G_H \to \mathrm{GL}_2(E)$ be 
the representation obtained by extending scalars to $E$.
When $p=2$, the set $\mathrm{IF}_{\rho}^{G_H}(E)$ has exactly $e(E/K)+1$ elements when $\sqrt{-3} \in E$, while it is a singleton when $\sqrt{-3} \notin E$.
On the other hand, assume $p\neq 2$, and assume $E$ 
is any algebraic extension of $\mathbb{Q}_p$.
Then $\mathrm{IF}_{\rho}^{G_H}(E)$ is a singleton, i.e., all representations $\rho': G_H \to \mathrm{GL}_2(\mathcal{O}_E)$, with $\rho'\otimes_{\mathcal{O}_E} E \cong\rho$, are 
$\mathrm{GL}_2(\mathcal{O}_E)$-conjugates.
\end{Theorem}

The dicyclic group can be defined as $G_D=\langle \mathtt{r}, \mathtt{p} \rangle$, where $\mathtt{r}=
\frac{1}{2}(1+\mathtt{q})$ and $\mathtt{p}$ and $\mathtt{q}$ are the standard generators of $\mathfrak{A}'=
\big(\frac{-3,-1}{\mathbb{Q}}\big)$.
This group has order $12$ and it is isomorphic to a certain semi-direct product $C_3 \rtimes C_4$.
Its integral $p$-adic representations are described by the next result, which is proved in \S \ref{subsection int rep max orders}.

\begin{Theorem}\label{main teo 3}
Let $E$ be a Galois extension of $\mathbb{Q}_p$ containing a 
quadratic extension, and let $\rho: G_D \to \mathrm{GL}_2(E)$ 
be the representation obtained by extending scalars to $E$.
When $p=3$, the set $\mathrm{IF}_{\rho}^{G_D}(E)$ has exactly $e(E/K)+1$ elements when $\sqrt{2} \in E$, while it is a singleton when $\sqrt{2} \notin E$.
Now, assume that $p\neq 3$, and let $E$ be an algebraic extension of $\mathbb{Q}_p$. Then, $\mathrm{IF}_{\rho}^{G_D}(E)$ is again a singleton.
\end{Theorem}

In order to exhibit the 
utility of twisted forms of Bruhat-Tits 
trees as a tool to compute local integral representations, in \S 
\ref{section rep ham} we study integral $2$-adic representation 
of $Q_8$ over quadratic extension of $\mathbb{Q}_2$. 
Note that they are all contained in 
$\Omega=\mathbb{Q}_2(\sqrt{-1},\sqrt{-3},\sqrt{2})$.
The faithful two dimensional representation of $Q_8$ is again 
absolutely irreducible.

\begin{Theorem}\label{main teo 4}
The group $Q_8$ has $26$ integral representations defined over 
the field $\Omega$ given above.
These representations are defined over intermediate fields 
$\Omega \supset L \supset \mathbb{Q}_2$ as follows 
(see Table \ref{table1} in \S \ref{section rep ham} 
for more details):
\begin{itemize}
    \item $4$ are defined over $L=\mathbb{Q}_2(\sqrt{\alpha})$, for $\alpha \neq -3$, while $2$ are defined over $L=\mathbb{Q}_2(\sqrt{-3})$, and
    \item $10$ are defined over $L=\mathbb{Q}_2(\sqrt{\alpha}, \sqrt{\beta})$, when $\sqrt{-3}\notin L$, while $6$ are defined over any other quartic subextension $L$.
\end{itemize}
\end{Theorem}

Since the order spanned by $Q_8$ is not maximal, this
result does not follow from Theorem \ref{main teo 1}. 
Moreover, Prop. \ref{tabla} in \S \ref{section rep ham} describes 
the $\mathcal{O}_{\Omega}$-representations of $Q_8$ simultaneously 
defined over pairs of intermediate subfields of $\Omega$.

Finally, in \S 
\ref{section int rep of ham over global fields}, we  
focus on the description of the global integral 
representations of $Q_8$.
Let $k$ be a number field and denote by $h_k(2)$ the 
cardinality of the 
maximal exponent-$2$ subgroup of the ideal class group
$\mathbf{Cl}_k$ of $k$.
By using Th. \ref{main teo 4} we prove next result, 
which describes the number of 
conjugacy classes of integral 
representations of $Q_8$ whenever they exist. 

\begin{Theorem}\label{main teo 5}
Let $N \equiv a \,\, (\text{mod } 8)$ be a positive integer, 
with $a\in \lbrace 1,2,3,5,6\rbrace$.
Assume $Q_8$ has an integral representation over 
$k=\mathbb{Q}(\sqrt{-N})$.
Then:
\begin{itemize}
\item[(a)] If $a=3$, then $\mathrm{IF}^{Q_8}_{\rho}(k)$
has $2h_k(2)$ elements,
\item[(b)] If $a\neq 3$ and the class 
$[\wp_{\mathbf{2}}]\in\mathbf{Cl}_k$ of 
the only dyadic maximal ideal $\wp_{\mathbf{2}}\subseteq
\mathcal{O}_k$ is a square, then 
$\mathrm{IF}^{Q_8}_{\rho}(k)$ has $4h_k(2)$ elements, and
\item[(c)] If $a\neq 3$ and the class 
$[\wp_{\mathbf{2}}]\in\mathbf{Cl}_k$ of 
the only dyadic maximal ideal $\wp_{\mathbf{2}}\subseteq
\mathcal{O}_k$ is not a square, then $\mathrm{IF}^{Q_8}_{\rho}(k)$ 
has either $h_k(2)$ or $3h_k(2)$ elements.
\end{itemize}
\end{Theorem}

Both cases in Th. \ref{main teo 5}(c) naturally appear in 
basic cases according to examples \ref{ex 3hk} and \ref{ex hk}.
Thus far, we have no way of distinguish one from 
the other without explicitly writing down a 
representation.
Analogous results could be proven for the Hurwitz group $G_H\cong \mathrm{SL}_2(\mathbb{F}_3)$ and the dicyclic group 
$G_D \cong C_2 \rtimes C_4$ by using the machinery developed in \S \ref{section int rep of ham 
over global fields} and previously extending the results of \S \ref{section rep ham} to the involved groups.

As for a generalization in a different direction,
we expect that the tools developed here inspire future 
studies on the local integral representations of other algebraic groups. Indeed, in the language of scheme theory, the idea is as follows:
Let $K$ be a local field as above, let $G$ be a 
finite group and let $\mathbb{G}$ be a reductive linear 
algebraic $K$-group.
Assume that $\mathbb{G}$ has an $\mathcal{O}_K$-model 
$\mathbb{G}^{\mathrm{int}}$, i.e., a $\mathcal{O}_K$-group 
scheme satisfying $\mathbb{G}^{\mathrm{int}} 
\otimes_{\mathcal{O}_K} K = \mathbb{G}$, or equivalently 
satisfying $\mathbb{G}^{\mathrm{int}}(A_K)= \mathbb{G}(A_K)$, 
for any $K$-algebra $A_K$, where $\mathbb{G}(A_K)$ denotes the (abstract) group of $A_K$-points of $\mathbb{G}$.
Then, the group 
$\mathbb{G}^{\mathrm{int}}(\mathcal{O}_K)$ of 
$\mathcal{O}_K$-points of $\mathbb{G}^{\mathrm{int}}$ 
is a compact subgroup of $\mathbb{G}(K)$ and moreover it is a maximal compact subgroup whenever $\mathbb{G}$ is semi-simple according to \cite[Prop. 8.2.1]{BT1} or \cite[Prop. 2.3.5]{loise2017}.
Let us consider a group homomorphisms 
$\rho: G \to \mathbb{G}(K)$.
The results described in \S \ref{branches and Int rep} can be 
rephrased by saying that the $\mathcal{O}_K$-conjugacy classes 
of group homomorphisms of the form 
$\rho': G \to \mathbb{G}^{\mathrm{int}}(\mathcal{O}_K)$ 
such that $\rho' \otimes_{\mathcal{O}_K} K \cong \rho$ can be 
indexed by suitable subcomplexes of the Bruhat-Tits building 
$\mathcal{X}(\mathbb{G},K)$ of $\mathbb{G}$ at $K$.
This can be easily extended to closely related
classical matrix groups such as $\mathrm{SL}_n$ or $\mathrm{PGL}_n$. Moreover, we expect that this statement also holds for other classical algebraic groups such as $\mathrm{Sp}_{2n}$, $\mathrm{SO}_n$ or $\mathrm{SU}_n$, by interpreting the vertex set of the corresponding Bruhat-Tits buildings in terms of lattices in a suitable symplectic, quadratic, or hermitian space.
See \cite{BT3,BT4} for details.
Moreover, we expect that the theory of twisted forms of the 
Bruhat-Tits tree of $\mathrm{SL}_2(K)$ can be eventually 
extended to the Bruhat-Tits building of 
$\mathrm{SL}_{n}(K)$, 
allowing the study of integral 
representations of finite subgroups of the
group of units of a (non-split) higher-rank central simple 
algebra.

\section{Integral representations, orders and trees}\label{branches and Int rep}

Let $G$ be a (finite) group, 
and let $\mathtt{K}$ be a local or global field.
Let $\mathcal{O}_{\mathtt{K}}$ be the ring of integers of $\mathtt{K}$.
Let $\mathrm{Rep}_{n}^G(\mathtt{K})$ be the set of 
$\mathrm{GL}_n(\mathtt{K})$-conjugacy 
classes of representations 
$\rho: G \to \mathrm{GL}_n(\mathtt{K})$,
which we assume is not empty. For $\rho \in \mathrm{Rep}_{n}^G(\mathtt{K})$, we keep the notation
$\mathrm{IF}_{\rho}^{G}(\mathtt{K})$,
introduced in \S \ref{section intro}, for its set of integral forms.
Let $\mathbb{O}$ be the $\mathrm{GL}_n(\mathtt{K})$-conjugacy class
of the subring $\mathbb{M}_n(\mathcal{O}_\mathtt{K})\subseteq
\mathbb{M}_n(\mathtt{K})$.
Assuming that $\mathtt{K}=K$ is a local field, $\mathbb{O}$ is exactly the set 
of maximal orders in $\mathbb{M}_n(K)$, i.e.,
maximal $\mathcal{O}_K$-lattices of $\mathbb{M}_n(K)$ 
that are rings 
under the multiplication law of $\mathbb{M}_n(K)$.
Next result describes 
$\mathrm{IF}_{\rho}^G(K)$ in terms 
of suitable subsets of $\mathbb{O}$.

\begin{Prop}\label{prop equiv rep and conj class}
Assume that $K$ is a local field.
There exists a bijection between $\mathrm{IF}_{\rho}^G(K)$
and the set of $\mathfrak{L}^*$-conjugacy classes of maximal orders 
$\mathfrak{D} \in \mathbb{O}$ containing $\rho(G)$, 
where $\mathfrak{L}$ is the centralizer of $\rho$. 
In particular, if $\rho$ is absolutely irreducible,
then $\mathrm{IF}_{\rho}^G(K)$ is in bijection with the set of
maximal orders $\mathfrak{D} \in \mathbb{O}$ containing $\rho(G)$.
\end{Prop}

\begin{proof}
We claim that each quotient is in bijection with the set of
$\mathrm{GL}_n(K)$-conjugacy classes of pairs $(\mathfrak{D},\rho')$
where $\mathfrak{D}\in\mathbb{O}$, and $\rho'\in[\rho]$ satisfies
$\rho'(G)\subseteq\mathfrak{D}$. 

On one hand, since $\mathrm{GL}_n(K)$ acts transitively on
the set of maximal orders, every such pair is in the orbit of a pair
of the form $(\mathfrak{D}_0,\rho'')$ with $\mathfrak{D}_0=
\mathbb{M}_n(\mathcal{O}_K)$ and 
$\rho''\in \mathrm{IF}_{\rho}^G(K)$.
Furthermore, two of the latter pairs, say $(\mathfrak{D}_0,\rho'')$
and $(\mathfrak{D}_0,\rho''')$, are conjugate precisely when
$\rho''$ and $\rho'''$ are conjugate by a matrix leaving
$\mathfrak{D}_0$ invariant. The normalizer of $\mathfrak{D}_0$
in $\mathrm{GL}_n(K)$ is $K^*\mathfrak{D}_0^*$ with scalar
matrices acting trivially, whence we conclude that
$\rho''\mapsto (\mathfrak{D}_0,\rho'')$ defines a bijection at class
level, as claimed.

On the other hand, since, by definition, $\mathrm{GL}_n(K)$ acts 
transitively on the conjugacy class $[\rho]$, every pair
of the form $(\mathfrak{D},\rho')$ has a conjugate of the form
$(\mathfrak{D}',\rho)$, with $\rho(G)\subseteq\mathfrak{D}'$.
Again, if two of these pairs are conjugate, the matrix realizing
this conjugation must leave $\rho$ invariant, so the result follows
as before.
\end{proof}

It follows from \cite[\S 6.5]{AbramenkoBrown} that $\mathbb{O}$ is
in bijection with the vertex set of a certain CW-complex called the 
Bruhat-Tits building of $\mathrm{SL}_n$.
When $n=2$, this complex is indeed a tree, called the Bruhat-Tits 
tree of $\mathrm{SL}_2(K)$. Here we give a construction
of the latter, in terms of the topology of $K$.
See~\cite[Chap. II, \S~1]{SerreTrees} or \cite{BT1} 
for details on the original definition.
Note that $K$ is a complete field with respect to the 
absolute value induced by a discrete 
valuation map $\nu_K: K \to \mathbb{R} \cup \lbrace 
\infty \rbrace$. In general, we assume that 
$\nu_K(K^*)=\mathbb{Z}\lambda$ for some real number
$\lambda$ that we choose later so that valuations can be 
conveniently normalized through various fields. We define 
an absolute value $|x|=c^{\nu_K(x)}$ for some real constant 
$c<1$ that we need not specify in this work. Metric and topology
on $K$ are defined via the absolute value. See \cite[\S II.1]{Serre-localfields} for details.

Indeed, we set $V$ as the set of closed 
balls in $K$, and $E$ as the set of pairs of 
balls $\{B,B'\}$ where one is a proper maximal 
sub-ball of the other.
The edge $e$ is called upwards from $B$ if $B\subseteq B'$, 
and downwards from $B$ otherwise.
The graph $\mathfrak{t}_K$ thus defined is a tree according to 
\cite[\S 4]{ArenasArenasContreras}.
Moreover, this tree is isomorphic to the Bruhat-Tits 
tree associated to $\mathrm{SL}_2(K)$.
We let $\mathcal{T}_K=(\mathfrak{t}_K,T_K)$ be
the associated topological graph. 

Fix a uniformizing parameter $\pi=\pi_K$ of $K$.
We denote by $B_{a,r}$ the closed ball in $K$ whose center is $a$ 
and whose radius is $|\pi_K^r|$. 
It is straightforward that any equivalence class of rays 
in $\mathfrak{t}_K$
has a representative $\mathfrak{r}$ satisfying one 
of the following conditions:
\begin{itemize}
\item $V(\mathfrak{r})=\left\lbrace B_{a,n} : 
n \in \mathbb{Z}_{\geq 0}  
\right\rbrace$, for certain $a \in K$, or
\item $V(\mathfrak{r})=\left\lbrace B_{0,-n} : 
n \in \mathbb{Z}_{\geq 0}  
\right\rbrace$.
\end{itemize}
In the first case, the visual limit of $\mathfrak{r}$ can be 
identified with $a \in K$, while, in the second, we identify 
it with the point at infinity $\infty$.
The same holds for the space $T_K$ or the topological 
graph $\mathcal{T}_K=(\mathfrak{t}_K,T_K)$.
This brief remark proves the following result:

\begin{Lemma}\label{lemma visual limits in tk}
The set of visual limits $\partial_{\infty}(\mathcal{T}_K)$ is in
bijection with the set of $K$-points $\mathbb{P}^1(K)$ of the 
projective line $\mathbb{P}^1$. \qed
\end{Lemma}

The group $\mathrm{GL}_2(K)$ acts classically on 
$\mathbb{P}^1(K)$, with scalar matrices acting trivially. 
Equivalently, there is a canonical action of the group 
$\mathcal{M}(K) \cong \mathrm{GL}_2(K)/ K^{*}$ of Moebius 
transformations with coefficients in $K$ on 
$\partial_{\infty}(\mathcal{T}_K)$.
Now, we show that this action extends naturally to an action of
the tree $\mathfrak{t}_K$, or the topological tree 
$\mathcal{T}_K$. In fact, if we remove the ball $B=B_{a,r}$, as a vertex, from $\mathfrak{t}_K$, 
the connected components of the remaining graph are in bijection with 
the elements of
$\mathbb{P}^1(\mathbb{K})$, where $\mathbb{K}=\mathcal{O}_K/\pi_K \mathcal{O}_K$ 
is the residue field of $K$.
More precisely, we can associate the element $\bar{b}\in\mathbb{K}$
to the connected component containing the
neighbor $B_{a+\pi_K^rb,r+1}$ of $B$
and associate $\infty$ to the connected component 
containing the smallest ball properly 
containing $B$. The visual limit of the former 
component is precisely
$B_{a+\pi_K^rb,r+1}$, while the visual limit of the 
latter component is the complement $B^c$. This gives 
us a decomposition into disjoint sets:
$$\mathbb{P}^1(K)=B^{\mathit{c}} \sqcup \coprod_{ \bar{b} 
\in \mathbb{K} } B_{a+b \pi_K^r, r+1}.$$
Every set in this decomposition is a ball or a ball complement.
Any decomposition into exactly $|\mathbb{K}|+1$ disjoint 
sets of this type corresponds to a ball.
The action by Moebius transformation permutes balls by permuting 
such decompositions. 
Likewise, neighboring balls can be characterized in 
terms of decompositions,
so this action preserves the structure of the tree.
See \cite{ArenasArenasContreras} or 
\cite{ArenasBravoExt} for details.
Alternatively, the action of a Moebius transformation 
$\mu\in\mathcal{M}(K)$
on a closed ball $B\subseteq K$ can be defined by:
\begin{enumerate}
\item[(a)] $\mu \bullet B=\mu(B)$, if the latter is a ball, while
\item[(b)] $\mu\bullet B$ is the smallest ball 
properly containing
$\mu(B^{\mathit{c}})$,
if $\mu(B^{\mathit{c}})$ is a ball.
\end{enumerate}

The group $\mathrm{GL}_2(K)$ acts on
the space $T_K$ via continuous maps. The action
of a Moebius transformation $\mu$ on an element $x \in T_K$
is denoted by $\mu \bullet x$.
This action is simplicial, in the sense that
is induced from an action on graphs.
In the sequel, we write $\mathcal{P}_K[v,w]
\subseteq\mathcal{T}_K$ for the finite line connecting 
$v,w \in V(\mathcal{T}_K)$.
Similar conventions apply for the ray  
$\mathcal{P}_K[v,a)\subseteq\mathcal{T}_K$
and the maximal path
$\mathcal{P}_K(a,b)\subseteq\mathcal{T}_K$, 
where $a,b\in\mathbb{P}^1(K)$. 
The graph subjacent to $\mathcal{P}_K[v,w]$
is denoted $\mathfrak{p}_K[v,w]$, and similar conventions
apply to rays and maximal paths.
The subindex $K$
is omitted when confusion is unlikely.

We define the peak in a maximal path $\mathcal{P}_K(a,b)$, 
when neither $a$ nor $b$ is $\infty$, as the vertex 
$B_{a, \nu(a-b)}=B_{b, \nu(a-b)}$, which is the vertex 
$B_{c,r}$ in the maximal
path with the smallest value of $r$. In the same spirit,
we call $B_{a,r+1}$ a neighbor below $B_{a,r}$,
and $B_{a,r}$ the neighbor above $B_{a,r+1}$. 
Note that every vertex has a unique neighbor above it and $|\mathbb{K}|$ neighbors below it.

We say that two homothety classes of $\mathcal{O}_K$-lattices
$[\Lambda]$ and  $[\Lambda']$ in $K\times K$ are neighbors when 
there 
exist representatives 
$\Lambda_0 \in [\Lambda]$ and $\Lambda_0' \in [\Lambda']$ such 
that 
$\Lambda_0 \subset \Lambda_0'$ and $ \Lambda_0'/\Lambda_0 
\cong \mathcal{O}_K/\pi_K \mathcal{O}_K$. It is not hard to see
that the latter condition implies both
$\Lambda'_0\subseteq\pi_K\Lambda_0$ and
$ (\pi_K\Lambda_0)/\Lambda'_0 \cong \mathcal{O}_K/\pi_K \mathcal{O}_K$,
so the preceding relation is symmetric.
In particular, we can define a graph whose vertices are the homothety
classes of lattices and whose edges correspond to the pairs
of neighboring classes as defined above.
It can be proved that the graph just defined is isomorphic to the
Bruhat-Tits tree $\mathfrak{t}_K$
(cf.~\cite[\S II.1, Th. 1]{SerreTrees}).
More explicitly, we can assume that this isomorphism takes
the ball $B_{a,r}$
to the homothety class containing
the lattice  $\Lambda_{a,r}=\left\langle
\binom{a}{1}, \binom{\pi_K^r}{0}
\right\rangle $.
The group $\mathrm{GL}_2(K)$ acts on homothety classes via 
$g \cdot [\Lambda]=[g(\Lambda)]$, for any $\mathcal{O}_K$-
lattice  $\Lambda \subset K \times K$ and any 
$g \in \mathrm{GL}_2(K)$. The correspondence between balls and 
homothety classes given above translates the latter into an 
action on the Bruhat-Tits tree that coincide with the actions 
via Moebius transformation defined earlier.

An order in $\mathbb{M}_2(K)$ is a lattice that is also a subring
of this matrix algebra. We say that an order is maximal when it 
fails to be strictly contained in any other order. For example, 
for any lattice $\Lambda\subseteq K \times K$ of rank $2$,  its 
endomorphism ring $\mathrm{End}_{\mathcal{O}_K}
(\Lambda)\subseteq\mathbb{M}_2(K)$  is a maximal order. It is 
well known that all maximal orders are conjugates, and therefore
they are endomorphism rings of lattices. Furthermore, two 
lattices have the same endomorphism ring precisely when
one is a multiple of the other. We conclude that maximal 
orders are in correspondence with homothety classes of lattices, 
whence we can redefine the Bruhat-Tits tree $\mathfrak{t}_K$ in 
terms of maximal orders. Indeed,
two maximal orders $\mathfrak{D}$ and
$\mathfrak{D}'$ are neighbors if, in some basis, they have the 
form $\mathfrak{D}=\mathfrak{D}_0$ and 
$\mathfrak{D}'=\mathfrak{D}_1$,
where
$$\mathfrak{D}_r=
\sbmattrix{\mathcal{O}_K}{\pi_K^r\mathcal{O}_K}
{\pi_K^{-r}\mathcal{O}_K}{\mathcal{O}_K}
.$$
More generally, for every pair of maximal orders 
$\mathfrak{D}$ and
$\mathfrak{D}'$ there exists a basis with respect to
which they have 
the form $\mathfrak{D}=\mathfrak{D}_0$ and 
$\mathfrak{D}'=\mathfrak{D}_n$,
where $n$ is the graph distance $\delta(v,v')$ between the 
corresponding vertices. The maximal orders corresponding
to vertices in the line $\mathcal{P}_K[v,v']$ have, in that basis,
the form  $\mathfrak{D}_r$ for $0\leq r\leq n$.

For every element $\mathbf{q} \in \mathbb{M}_2(K)$, we denote 
by $\mathfrak{s}_K(\mathbf{q})$
the largest subgraph of $\mathfrak{t}_K$ whose vertices correspond
to maximal orders containing $\mathbf{q}$. Such graph is 
called the branch of $\mathbf{q}$, and it is a connected graph,
since every maximal order
corresponding to a vertex in $\mathcal{P}_K[v,v']$, as above,
contains $\mathfrak{D}\cap\mathfrak{D}'$.
When $\mathbf{q}$ is not contained in any maximal order, i.e., 
when $\mathbf{q}$ is not integral over $\mathcal{O}_K$, then 
the branch $\mathfrak{s}_K(\mathbf{q})$ is empty.
We write $\mathcal{S}_K(\mathbf{q})\subseteq\mathcal{T}_K$
for the associated topological graph.

The rest of this section is devoted to 
giving a description for all
branches of quaternions.
See \cite{ArenasBranches}, \cite{ArenasBravoExt} or 
\cite{ArenasBravoCar2} for more details.
Let $\mathcal{F}_0\subseteq\mathcal{T}_K$ be the topological
subtree whose vertices are the balls of radius $1$ or more.

\begin{Lemma}\label{nilbranches}\cite[Prop. 4.2]{ArenasBranches}
If the two dimensional algebra $K(\mathbf{q})$ 
contains a nilpotent element, then 
$\mathcal{S}_K(\mathbf{q})=\gamma \bullet \mathcal{F}_0 $, 
for some Moebius transformation $\gamma$. 
\end{Lemma}

Otherwise, the algebra $K(\mathbf{q})$ is an \'etale algebra, i.e., either $K(\mathbf{q})$ is isomorphic to a quadratic extension field of $K$ or to $K \times K$.
In this case,
a full description of the graph $\mathcal{S}_K(\mathbf{q})$ 
is given by the following three lemmas. 
Before stating these results, 
we need an additional definition. Let $\mathcal{C}$ be a
topological subtree 
of $\mathcal{T}_K$. A tubular neighborhood $\mathcal{C}^{[n]}$ 
is the smallest subtree of $\mathcal{T}_K$ containing all 
vertices at a distance $s \leq n$ from the vertices in $\mathcal{C}$. The first two lemmas describe the branch when 
$\mathcal{O}_K[\mathbf{q}]$ is
the maximal integral ring of the \'etale algebra $K(\mathbf{q})$.
Note that any order in $K(\mathbf{q})$ can be written as $\mathcal{O}_K[\alpha \mathbf{q}]$,
where $\mathcal{O}_K[\mathbf{q}]$ is the maximal ring of integers.
This might require to adjust the generators as illustrated in 
the examples.

\begin{Lemma}\label{splitbranches}\cite[Prop. 4.1]{ArenasBranches}
If the ring $\mathcal{O}_K[\mathbf{q}]$ is isomorphic to 
$\mathcal{O}_K\times \mathcal{O}_K$,
then $\mathcal{S}_K(\mathbf{q})=
\gamma\bullet\mathcal{P}_K(0,\infty)$, 
for some Moebius transformation $\gamma$. 
\end{Lemma}

\begin{Lemma}\label{idbranches}\cite[Prop. 4.1]{ArenasBranches}
If $\mathcal{O}_K[\mathbf{q}]$ is isomorphic to $\mathcal{O}_E$, 
for some extension field $E/K$, then 
$\mathcal{S}_K(\mathbf{q})$ is a 
vertex or an edge, where the latter holds 
precisely when $E/K$ ramifies.
\end{Lemma}

\begin{Lemma}\label{lemma contraidos}\cite[Prop. 2.4]{ArenasBranches}
For any $\alpha \in \mathcal{O}_K$, and for any element
$\mathbf{q}\in\mathbb{M}_2(K)$ with a non-empty branch,
the identity 
$\mathcal{S}_K(\alpha \mathbf{q})=
\mathcal{S}_K(\mathbf{q})^{[\nu(\alpha)]}$
holds.
\end{Lemma}

The following examples illustrate how these lemmas allow us to compute branches of orders:

\begin{Example}
    If $\mathcal{S}_K(\mathbf{q})=\mathcal{F}_0$ is as 
    in Lemma \ref{nilbranches}, then
    $\mathcal{S}_0(\pi_K \mathbf{q})$ is the graph whose 
    vertices are the
    balls of radius $|\pi_K|$ or more, according to the last lemma.
    Note that this branch is in the $\mathcal{M}(K)$-orbit of
    $\mathcal{F}_0$, which is consistent with the fact that the 
    orders $\mathcal{O}_K[\mathbf{q}]$ and 
    $\mathcal{O}_K[\pi_K \mathbf{q}]$ are conjugates.
\end{Example}

\begin{Example}
    If $L=\mathbb{Q}_2$ and $\mathbf{q}=\sbmattrix0{20}10$, then
    we can write $\mathbf{q}=2(2\mathbf{u}+\mathbf{1})$,
    where $\mathbf{u}$ 
    has the minimal polynomial
    $x^2+x-1$, and $\mathbf{1}$ is the identity matrix.
    Therefore $\mathbf{u}$ generates the ring of integers
    of an unramified quadratic extension. In this case
    $\mathcal{S}_K(\mathbf{q})$ is a graph with 10 total vertices.
    One central vertex, three neighbors of the central vertex
    and six leaves.
\end{Example}

\begin{Example}\label{exa47}
    If $\mathbf{q}=\sbmattrix u00{u+\epsilon}$, then
    we can write $\mathbf{q}=u\mathbf{1}+\epsilon\mathbf{d}$, where
    $\mathbf{d}=\sbmattrix 0001$ generates an algebra isomorphic to
    $\mathcal{O}_K\times \mathcal{O}_K$. We conclude that
    $\mathcal{S}_K(\mathbf{q})=\mathcal{P}^{[\nu(\epsilon)]}$,
    where $\mathcal{P}$ is a maximal path. Since $\epsilon$
    is the difference of the eigenvalues of $\mathbf{q}$, we
    describe $\mathcal{S}_K(\mathbf{q})$ as a tubular neighborhood
    of a maximal path whose width is the valuation of the 
    difference of the eigenvalues. By taking conjugates, the same is true for any matrix with eigenvalues in the field $K$.
    The visual limits of the maximal path can be characterized as
    the invariant projective points of the associated Moebius
    transformation, which is $\eta(z)=\frac u{u+\epsilon}z$
    for the diagonal matrix $\mathbf{q}$ above.
\end{Example}


In general, we write 
$\mathcal{S}_K(\mathbf{q}_1,\dots,\mathbf{q}_n)=
\bigcap_{i=1}^n\mathcal{S}_K(\mathbf{q}_i)$. 
This is the graph whose vertices correspond
precisely to the maximal orders containing 
each of the matrices $\mathbf{q}_i$,
or equivalently, the group or order that these matrices generate.
We use the notations $\mathcal{S}_K(\mathfrak{H})=
\mathcal{S}_K(\mathbf{q}_1,\dots,\mathbf{q}_n)$, where 
$\mathfrak{H}=\mathcal{O}_K[\mathbf{q}_1,\dots,\mathbf{q}_n]$ 
is the generated order, or $\mathcal{S}_K(G)=
\mathcal{S}_K(\mathbf{q}_1,\dots,\mathbf{q}_n)$, if 
$G=\langle \mathbf{q}_1,\dots,\mathbf{q}_n \rangle$ is the 
generated multiplicative group.
We say that two vertices in $\mathfrak{t}_K$ or $\mathcal{T}_K$
are of the same type if the distance between them belongs to
$2\nu_K(K^*)=2\lambda\mathbb{Z}$.

\begin{Lemma}\label{lemma branch and MT}
If $\mathbf{q}$ is a unit, i.e., if it is integral 
over $\mathcal{O}_K$
and its determinant is a unit, then the elements of
$\mathcal{S}_K(\mathbf{q})\subseteq \mathcal{T}_K$ are precisely
the invariant points of the Moebius transformation $\chi$ 
corresponding to $\mathbf{q}$.
\end{Lemma}

\begin{proof}
Since the action of Moebius transformations is simplicial
on the topological graph $\mathcal{T}_K$, an element
$x\in \mathcal{T}_K$ is invariant if and only if one of the following
conditions holds. 
\begin{itemize}
    \item The endpoints of the edge containing $x$ are invariant.
    \item $\chi$ transposes the endpoints of the edge $e$ containing
    $x$, and $x$ is the midpoint of $e$.
\end{itemize}
The second alternative cannot hold since the determinant of
$\mathbf{q}$ is a unit, and therefore the action on $\mathcal{T}_K$ 
induced by the transformation $\chi$ preserves the type of
vertices \cite[\S II.1.3, Lemma 1]{SerreTrees}. 
Therefore, we just need to prove the result for vertices. 
Indeed, let us denote by $v_{[\Lambda]}$ the vertex 
corresponding to the class $[\Lambda]$ of the full rank 
lattice $\Lambda$. The set of invariant vertices is 
$$\lbrace v_{[\Lambda]} : \gamma \bullet v_{[\Lambda]}=v_{[\Lambda]}\rbrace
=\lbrace v_{[\Lambda]} : [\mathbf{q}\Lambda] = [\Lambda]\rbrace.$$
Recall that $[\mathbf{q} \Lambda] = [\Lambda]$ means 
$\mathbf{q} \Lambda 
= a\Lambda$ for $a\in K$. Since the determinant of $\mathbf{q}$
is a unit, 
multiplication by $\mathbf{q}$ preserves the Haar measure in $K^2$.
Since a lattice is a compact open set, its measure is finite 
and non-null,
whence $a$ must be a unit. On the other hand, by definition
$$V\big(\mathcal{S}_K(\mathbf{q})\big)=
\lbrace v_{[\Lambda]} : \mathbf{q} \in\mathfrak{D}_{\Lambda}\rbrace=
\lbrace v_{[\Lambda]} : \mathbf{q}\Lambda\subseteq\Lambda\rbrace.$$
Again, since $\mathbf{q}$ preserves the Haar measure, 
the identity defining 
the last set is equivalent to $\mathbf{q}\Lambda=\Lambda$,
since otherwise $\Lambda\smallsetminus (\mathbf{q}\Lambda)$
is a compact clopen set, and therefore it has positive measure. 
The result follows.
\end{proof}

\section{Field extensions and Bruhat-Tits trees}
\label{subsection ext BTT}

In what follows, we denote by $K$ a local field with valuation 
$\nu_K: K \twoheadrightarrow \mathbb{Z} \cup \lbrace \infty \rbrace$,
i.e., $\lambda=1$ in the notations of \S\ref{branches and Int rep}. We let $L$ be a finite
extension  of $K$. Let $e(L/K)$ be the ramification index of $L/K$.
We normalize the valuation
$\nu_L$ of $L$ by setting $\nu_L(\pi_L)=\frac{1}{e(L/K)}$, 
for some (equiv. any) uniformizing parameter $\pi_L$.
In other words, we assume that $\nu_L(L^*)=\frac 1{e(L/K)} \mathbb{Z}$.
Note that, with this convention,
we get $\nu_K(x)=\nu_L(x)$ for any element $x\in K$, so we can 
write $\nu$ unambiguously for either $\nu_K$ or $\nu_L$.
We adopt this convention in all that follows.
Since balls are fully determined by its radius and one center, 
the vertex set $V_K=V({\mathcal{T}_K})$ of $\mathcal{T}_K$ 
can be naturally identified with a subset of the vertex set
$V_L=V({\mathcal{T}_L})$  of $\mathcal{T}_L$.
In what follows, an element $v\in V({\mathcal{T}_L})$ is 
called a $K$-vertex, 
if it corresponds to a vertex in $V({\mathcal{T}_K})$.
This definition extends naturally to intermediate fields.
Following \cite{ArenasBravoExt}, vertices in $\mathcal{T}_L$ that are  
not $K$-vertices are called here $K$-ghost vertices.
Since neighboring vertices in $\mathcal{T}_K$ correspond to 
$K$-vertices that are connected by a path with $e(L/K)$ edges  in 
$\mathcal{T}_L$, the identification of vertices in $\mathcal{T}_K$
with $K$-vertices in $\mathcal{T}_L$
cannot  define a simplicial map unless $L/K$ is 
an unramified extension.  In order to solve this problem, we introduce
a normalized distance function on 
$V(\mathcal{T}_L)\subseteq\mathcal{T}_L$, by the formula:
$$d^*(v,v')=\frac{1}{e(L/K)}d(v,v'),$$
where $d$ is the usual distance on graphs defined in 
\S\ref{subsection conventions}. This extends naturally to
a distance in the topological graph $\mathcal{T}_L$ that is 
used in the sequel. The normalized graph distance is thus coherent 
with the corresponding definition with intermediate fields,
so we do not need to use suffixes when considering $\mathcal{T}_K$ and 
$\mathcal{T}_L$ simultaneously. Furthermore, the identification of 
vertices induces a (non-simplicial) map $\iota_{L/K}:\mathcal{T}_K
\to \mathcal{T}_L$ that
is a closed embedding, and an isometry with its image.
We identify the space $T_K$ with a subspace of $T_L$ in this way.
In this sense we write $\mathcal{T}_K\subseteq\mathcal{T}_L$
in what follows. The same holds for any intermediate field 
$K \subseteq F \subseteq L$. Under this identification, the distance 
between two $L$-vertices that are also $F$-vertices, for some 
intermediate field $F$, is unambiguously defined. A real number is said 
to be defined over $F$ if it is a multiple of the valuation 
of the corresponding uniformizer $\nu(\pi_F)$. We apply this definition
to both, valuations of elements or distance between points in
$\mathcal{T}_L$. Next result is evident but useful:

\begin{Lemma}\label{lemma defined over F}
Let $v$ and $v'$ be two $F$-vertices in $\mathcal{T}_L$.
Then, a point $v''\in \mathcal{P}_L[v,v']$ is an $F$-vertex 
if and only if its distance to either $v,v'$ is defined over $F$.  \qed
\end{Lemma}

A finite line $\mathcal{P}_L[v,v']$ between two $F$-vertices 
$v,v'$ is said to be 
defined over  $F$. The same applies to maximal paths and rays, 
replacing endpoints by visual limits if needed.
Note that visual limits of $\mathcal{T}_F\subseteq \mathcal{T}_L$ 
are identified with the corresponding visual limits of $\mathcal{T}_L$. 
This defines an embedding from the set $\partial_{\infty}(\mathcal{T}_K)$ into
$\partial_{\infty}(\mathcal{T}_L)$ that coincides with the usual 
embedding of $\mathbb{P}^1(F)$ into $\mathbb{P}^1(L)$.

Let $\mathcal{P}_L(a,b)$ be the unique maximal path in 
$\mathcal{T}_L$ whose visual limits are 
$a,b \in \mathbb{P}^1(L)$.
When the extension $L/K$ is Galois, the group $\mathcal{G}_L=\mathrm{Gal}(L/K)$ acts
naturally on the set $\mathbb{P}^1(L)$.
This action extends to the topological tree $\mathcal{T}_L$,
by exactly the same reason as the action of Moebius transformations,
i.e., by considering the action on partitions of the projective line.
In the sequel, this Galois action is called standard, as opposed to the
twisted actions defined in next section.
The standard Galois action satisfies the following compatibility relation on maximal paths (cf. \cite[\S 3]{ArenasBravoExt}):
$$\sigma\Big(\mathcal{P}_L(a,b)\Big)=
\mathcal{P}_L\Big(\sigma(a),\sigma(b)\Big),$$
where $\sigma\in\mathcal{G}_L$ and where $a,b\in\mathbb{P}^1(L).$
We also have a compatibility property with the action of the group $\mathcal{M}(L)$, of Moebius transformations, on $\mathcal{T}_L$.
Indeed, we have:
\begin{equation}\label{eq comp actions}
\sigma(\mu \bullet x)=\sigma(\mu)\bullet \sigma(x),
\end{equation}
for any $\sigma \in \mathcal{G}_L$, $\mu\in\mathcal{M}(L)$ and any 
$x\in \mathcal{T}_L$. 


Let $L/K$ be a quadratic (galois) extension.
It follows from \cite[\S2.6.1]{Titsreductivegroups} or \cite{Rousseau}, that the invariant space $T_L^{\mathcal{G}_L}$ is strictly  
larger than the space $T_K$ corresponding to
$\mathcal{T}_K$ whenever $L/K$ is wildly ramified.
In general, the invariant subspace of a group action
is not always a topological subgraph. It can also
be the center of an edge.

\begin{Lemma}
    If the action of a group $\mathcal{H}$ on a topological tree
    $\mathcal{C}$ has an invariant vertex, then the set
    of invariant points corresponds to a subgraph. In particular,
    the invariant space $T_L^{\mathcal{G}_L}$ corresponds to a
    subtree of $\mathcal{T}_L$.
\end{Lemma}

\begin{proof}
    Let $v\in V(\mathcal{C})$ be an invariant vertex. If 
    $x\in\mathcal{C}$ is an invariant point inside an edge, 
    then so is every point in the unique 
    path joining $x$ to $v$. One of them is an endpoint of the edge
    containing $x$, and therefore every point in that edge is invariant.
    This proves that the invariant subspace is a subgraph.
    To see that it is a subtree, we need to see that it is connected.
    But this again follows from the fact that the unique path between
    invariant points is also invariant.
    The result follows. 
\end{proof}

We write  $\mathcal{T}_L^{\mathcal{G}_L}$ for the invariant topological subgraph. Now, assume that $\mathrm{char}(K)\neq 2$.
In the sequel, we denote by $\delta_K(a)$ the quadratic defect of $a \in K$, i.e., the smallest fractional ideal in $K$ spanned by an element of the form $a-b^2$, with $b\in K$ (cf. \cite{omeara}).
Next result describes the space $\mathcal{T}_L^{\mathcal{G}_L} 
\smallsetminus \mathcal{T}_K$ in terms of certain quadratic defects.

\begin{figure}
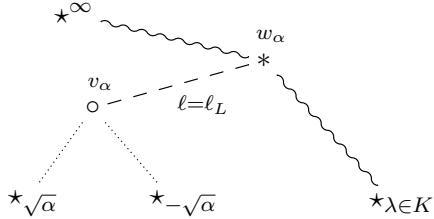

\[
\xygraph{!{<0cm,0cm>;<0.9cm,0cm>:<0cm,0.9cm>::}
!{(-2.4,1.8) }*+{}="i"
!{(-1,0.8) }*+{\circ}="b"
!{(-1,1.1) }*+{\hphantom{x}^{v_{\alpha}}}="bim"
!{(1.5,1.5) }*+{\ast}="f"
!{(1.5,1.8) }*+{\hphantom{x}^{w_{\alpha}}}="fim"
!{(0.5,0.8) }*+{\hphantom{x}^{\ell=\ell_L}}="mim"
!{(3.4,-0.6) }*+{\hphantom{x}\star_{\lambda \in K}}="a"
!{(-1.4,2.2) }*+{\hphantom{x}\star^{\infty}}="i"
!{(0.2,-0.6)}*+{\hphantom{x}\star_{-\sqrt{\alpha}}}="g"
!{(-2,-0.6)}*+{\hphantom{x}\star_{\sqrt{\alpha}}}="e"
"b"-@{--}"f" "a"-@{~}"f" "b"-@{.}"e" "b"-@{.}"g" "f"-@{~}"i"}
\]
\caption{The trees $\mathcal{T}_K \subseteq \mathcal{T}_L^{\mathcal{G}_L} \subset \mathcal{T}_L$. In this figure $w_{\alpha}=B_{\alpha, \frac{1}{2}\nu\big(\delta_K(\alpha)\big)}$ is the closest point of $\mathcal{T}_K$ to $\mathcal{P}_L(\sqrt{\alpha},-\sqrt{\alpha})$, while $v_\alpha=B_{\sqrt{\alpha}, \nu(2\sqrt{\alpha})}$ is the unique $\mathcal{G}_L$-invariant point of $\mathcal{P}_L(\sqrt{\alpha},-\sqrt{\alpha})$.
}\label{figura inv subtrees}
\end{figure}

\begin{Prop}\label{prop hairs}
Let us write $L=K(\sqrt{\alpha})$, with $\alpha \in \mathcal{O}_K$.
The invariant subspace $\mathcal{T}_L^{\mathcal{G}_L}$ contains precisely 
the points $x$ whose distance $d=d(x,\mathcal{T}_K)$ from $\mathcal{T}_K$
satisfies $d\leq \ell_{L}:=
\nu(2\sqrt{\alpha})-\frac{1}{2}\nu\big(\delta_K(\alpha)\big)$ from $\mathcal{T}_K$.
\end{Prop}

\begin{proof}
Let $B$ be a vertex in $\mathcal{T}_L^{\mathcal{G}_L}$, which we write as $B=B_{z,r}$, for some $z\in L$ and some $r\in \mathbb{Z}$.
Since $z=x+y\sqrt{\alpha}$, with $x,y \in K$, we have $y \neq 0$ or $v\in \mathcal{T}_K$. In the first case, up to applying the Moebuis transformation $z \mapsto \frac{z-x}{y}$, we assume that $z=\sqrt{\alpha}$.
Since a generator of $\mathcal{G}_L$ exchanges
the visual limits of the maximal path 
$\mathcal{P}_L(\sqrt\alpha,-\sqrt\alpha)$, the unique $\mathcal{G}_L$-invariant point of the path is $B_{\sqrt{\alpha}, \nu(2\sqrt{\alpha})}$. Moreover, by exactly the same reason, the vertex $B_{\sqrt{\alpha}, r}$ is $\mathcal{G}_L$-invariant exactly when $r\leq \nu(2\sqrt{\alpha})$. See Fig. \ref{figura inv subtrees}(A).
Now, note that a vertex $B \in \mathcal{T}_L$ is a point in the space 
corresponding to $\mathcal{T}_K$ (but not necessarily a vertex
in the latter) exactly when 
$B=B_{\lambda,r}$, for some $\lambda\in K$ and some $r\in \nu(L^*)$.
In other words, the vertex $B_{\sqrt{\alpha}, r}$ belongs to $\mathcal{T}_K$ exactly when
$B_{\sqrt{\alpha}, r}\cap K\neq \varnothing$,
since any $\lambda\in B_{\sqrt{\alpha}, r}\cap K$ is a center of this ball. Equivalently,
we have
$r\leq \nu(\sqrt{\alpha}-\lambda)\leq 
\frac{1}{2}\nu\big(\delta_K(\alpha)\big)$,
from the definition of quadratic defect.
In particular, $B$ does not belong to $T_K$ precisely when $\nu(2\sqrt{\alpha}) \geq r> \frac{1}{2}\nu\big(\delta_K(\alpha)\big)$.
Therefore, the result holds for vertices.
Now the result follows since the invariant set 
$\mathcal{T}_L^{\mathcal{G}_L}$ is a subtree.
\end{proof}

\begin{Remark}
Prop. \ref{prop hairs} can be extended to even characteristic local fields by using the Artin-Schreier defect introduced in \cite{ArenasBravoCar2} and some results contained in the same article.
We also believe that it might be possible to extend
Prop. \ref{prop hairs} to the higher rank context, i.e., for the Bruhat-Tits building of $\mathrm{SL}_n$.
We intend to address this possibility in a future work.
\end{Remark}

\section{Twisted forms of Bruhat-Tits trees}\label{subsection forms of BTT}

Let $L/K$ be a (finite) Galois field extension with 
Galois group $\mathcal{G}=\mathrm{Gal}(L/K)$.
An $L/K$-form of $\mathcal{T}_L$ is a pair 
$\widehat{\mathcal{T}}=(\mathcal{T}_L,*)$, where
$\mathcal{T}_L$ is the Bruhat-Tits tree as a 
topological graph, while $*$ is a simplicial 
$\mathcal{G}$-action. We refer to $*$ as the 
Galois action on $\widehat{\mathcal{T}}$. Note that $*$ might
coincide with the usual $\mathcal{G}$-action,
in which case $\widehat{\mathcal{T}}$ is called
the trivial form, an denoted, by abuse of notation,
as $\mathcal{T}_L$. An isomorphism between two of these forms
is a $\mathcal{G}$-equivariant simplicial homeomorphism.
We denote by $\mathcal{F}(L/K)$ 
the set of all $L/K$-forms of $\mathcal{T}_L$, or equivalently, 
the set of simplicial $\mathcal{G}$-actions on $\mathcal{T}_L$.
We also use the notation $\mathcal{E}(L/K)$ for the set 
of isomorphism classes of such forms.
As usual, we consider it as a pointed set where the class of
$\mathcal{T}_L$, the trivial form, is the distinguished element. 
 
The usual $\mathcal{G}$-action induces an action of $\mathcal{G}$ 
on the group $\mathrm{Simp}(\mathcal{T}_L)$ of simplicial 
automorphisms of $\mathcal{T}_L$. Indeed, for each 
$\sigma \in \mathcal{G}$ and each 
$\phi \in \mathrm{Simp}(\mathcal{T}_L)$, we define 
${}^\sigma\phi \in \mathrm{Simp}(\mathcal{T}_L)$ by setting 
${}^\sigma\phi(x)=\sigma \Big( \phi \big(\sigma^{-1}(x)\big) \Big)$, 
for all $x \in T_L$.
For each $1$-cocycle $a\in Z^1=Z^1\big(\mathcal{G}, 
\mathrm{Simp}(\mathcal{T}_L)\big)$, 
viewed as a tuple $a=(a_\tau)_{\tau\in\mathcal{G}}$ with
$a_\tau\in\mathrm{Simp}(\mathcal{T}_L)$, we have 
an $L/K$-form $\mathcal{T}=\mathcal{T}_{a}$ 
with the Galois action defined by $\tau* 
x:= a_{\tau}\big(\tau (x)\big)$, for all $\tau \in \mathcal{G}$ and all
$x\in  T_L$. We claim that each class in $\mathcal{E}(L/K)$ can be 
represented by a form $\mathcal{T}_{a}$, for a certain 
$a\in Z^1$, as in Prop. 33 in \cite[\S I.5]{Serre-CohomologieGaloisienne}.
Since our setting is not included in the theory described by Serre
neither in \cite[\S I.5]{Serre-CohomologieGaloisienne},
nor in \cite[\S III.1]{Serre-CohomologieGaloisienne},
we provide a proof of the following result:

\begin{Prop}\label{prop cohgal for trees}
There exists an explicit pointed bijective correspondence between the
cohomology set $H^1\big(\mathcal{G}, \mathrm{Simp}(\mathcal{T}_L)\big)$
and the pointed set $\mathcal{E}(L/K)$ sending every cocycle class
$[a]$ to the isomorphism class $[\mathcal{T}_{a}]$ of the corresponding 
form $\mathcal{T}_{a}$. 
\end{Prop}

\begin{proof}
  The fact that the formula $\tau*x= a_{\tau}\big(\tau (x)\big)$ defines 
  an action is an straightforward standard computation. This action
  is simplicial, since both the Galois group
  and $\mathrm{Simp}(\mathcal{T}_L)$ act on the tree via simplicial maps.
  This proves that $\mathcal{T}_{a}$ is indeed a $L/K$-form with our definition.

  The map $[a]\mapsto [\mathcal{T}_{a}]$ is injective since a 
  $\mathcal{G}$-equivariant homeomorphism $f$ between two forms 
  $\mathcal{T}_{a}$ and $\mathcal{T}_{b}$ satisfies
  $b_{\tau}\Big(\tau \big( f(x)\big)\Big)=
  f\Big(a_{\tau}\big(\tau (x)\big)\Big)$ for every point $x$ in the tree.
  This means that $f^{-1}\circ b_\tau\circ {}^{\tau} f=a_\tau$, and this 
  proves that the cocycles are cohomologous. On the other hand, for every
  form $\widehat{\mathcal{T}}$, the corresponding action 
  $x \mapsto \tau * x$ defines a cocycle via 
  $a_\tau(x)=\tau*\big(\tau^{-1}(x)\big)$, and this proves that
  the map is surjective.
\end{proof}

A Severi-Brauer variety of dimension one over $L/K$ is a 
$K$-variety which becomes isomorphic to $\mathbb{P}^1$ over $L$.
When $\mathrm{char}(K) \neq 2$, they are the conics of the form 
$\lbrace [x,y,z] \in \mathbb{P}^2\lvert z^2=ax^2+by^2 \rbrace$, where 
$a,b \in K$ and  $ax^2+by^2-z^2$ has a root over $L$.
The set $\mathrm{SB}_1(L/K)$ of isomorphism classes of such 
varieties form a pointed set 
with $[\mathbb{P}^1]$ as the distinguished point.
Let $\mathrm{Aut}(\mathbb{P}^1)$ be the group of automorphisms,
as a variety, of $\mathbb{P}^1$.
The group $\mathcal{G}$ acts canonically on $\mathbb{P}^1(L)$ as 
$\sigma \cdot [x:y]=[\sigma(x):\sigma(y)]$, for 
$x,y\in \mathbb{A}^1(L)$ and $\sigma \in \mathcal{G}$.
Thus $\mathcal{G}$ acts on $\mathrm{Aut}(\mathbb{P}^1)(L)$ via 
${}^{\sigma}\phi([x:y])=\sigma \cdot
\big( \phi ( \sigma^{-1} \cdot [x:y] \big)$.
More concretely, note that $\mathrm{Aut}(\mathbb{P}^1)(L)
\cong\mathcal{M}(L)\cong\mathrm{PGL}_2(L)$, whence
the latter action corresponds to the canonical action of 
$\mathcal{G}$ on coefficients of 
Moebius transformations or matrices.
Next result, where we write 
$H^1(\mathcal{G}, G)=H^1\big(\mathcal{G}, G(L)\big)$
for any algebraic group $G$, describes $\mathrm{SB}_1(L/K)$ in terms
of cohomology sets:

\begin{Prop}\cite[\S III.1]{Serre-CohomologieGaloisienne}
There exists an explicit bijective correspondence between the 
pointed sets $\mathrm{SB}_1(L/K)$ and $H^1(\mathcal{G}, \mathrm{PGL}_2)$
that preserves the distinguished element. Furthermore, this map
sends an isomorphism class of Severi-Brauer varieties $[V]$ to the 
class of the cocycle defined by $a_\tau(x)=\tau*\big(\tau^{-1}(x)\big)$, 
when we identify any representative $V$ of $[V]$ with $\mathbb{P}^1$ via
a non-equivariant map, so that $x\mapsto \tau(x)$ and $x\mapsto \tau* x$
denote the Galois action on $\mathbb{P}^1$ and $V$, respectively.
\end{Prop}

As every Moebius transformation defines
a simplicial automorphism of $\mathcal{T}_L$, there is a map of 
$\mathcal{G}$-groups from $\mathrm{PGL}_2(L)$ to 
$\mathrm{Simp}(\mathcal{T}_L)$. Passing to cohomology, we can 
associate a class $\left[\widehat{\mathcal{T}}\right]=
\left[\widehat{\mathcal{T}}\right](\mathbb{V})\in\mathcal{E}(L/K)$ 
to every isomorphism class $[\mathbb{V}] \in\mathrm{SB}_1(L/K)$.
The set of $L$-points $\mathbb{V}(L)$ is recovered, as a 
$\mathcal{G}$-set, as the visual limit 
$\partial_{\infty}(\widehat{\mathcal{T}})$, i.e., the visual limit 
$\mathbb{P}^1(L)$ of $\mathcal{T}_L$ with an action
induced from the Galois action on $\widehat{\mathcal{T}}$. 
Next result is now straightforward:

\begin{Lemma}\label{lemma visual limits as severi-brauer varieties}
For each Severi-Brauer variety $\mathbb{V}$ as above, 
the $\mathcal{G}$-invariant visual limits of the
associated $L/K$-form $\mathcal{T}=\mathcal{T}(\mathbb{V})$
are in correspondence with the set $\mathbb{V}(K)$ of $K$-points.
In particular $\mathcal{T}$ has no invariant visual limit if
$\mathbb{V}$ is non-trivial.\qed
\end{Lemma}

Denote by $\mathfrak{Alg}[L]$ the category of $L$-algebras.
Let $\mathfrak{A}$ be a $4$-dimensional central simple $K$-algebra.
Then, for any even degree extension $L$ of $K$ we have
$\mathfrak{A}_L \cong \mathbb{M}_2(L)$.
Let $\mathrm{Au}$ be the algebraic group whose group of $L$ points is
$\mathrm{Au}(L)=\mathrm{Aut}_{\mathfrak{Alg}[L]}\big(\mathbb{M}_2(L)\big)$.
Note that there is an isomorphism of algebraic groups 
$\mathrm{Au}\cong\mathrm{PGL}_2$, defined over $K$, 
but the elements of $\mathrm{Au}(L)$ are usually denoted as 
functions in the sequel to avoid confusion.
Every function $f\in\mathrm{Au}(L)$ corresponds to
a matrix $t_f$ in a way that $f(x)=t_fxt_f^{-1}$ for
$x\in\mathbb{M}_2(L)$. Note that two matrices defining
the same element in $\mathrm{PGL}_2(L)$ correspond to 
the same function. We make frequent use of this explicit
isomorphism in what follows.  Furthermore, the set 
$H^1\big(\mathcal{G}, \mathrm{Au}\big)\cong
H^1\big(\mathcal{G}, \mathrm{PGL}_2\big)$ is in 
correspondence with the pointed set of isomorphism classes of 
$K$-quaternion algebras $\mathfrak{A}$ that become isomorphic 
to $\mathbb{M}_2(L)$ over $L$.
For any such algebra $\mathfrak{A}$, there exists a variety
$\mathrm{Triv}=\mathrm{Hom}(\mathfrak{A},\mathbb{M}_2)$ 
whose set of $E$-points
is the set of homomorphisms, in $\mathfrak{Alg}[E]$, 
from $\mathfrak{A}_E$ 
to $\mathbb{M}_2(E)$, for any extension
$E$ of $K$. This variety always has an $L$ point,
but $\mathrm{Triv}(K)=\varnothing$, unless 
$\mathfrak{A}\cong\mathbb{M}_2(K)$.

\begin{Example}
Let $\mathbb{V}$ be the Severi-Brauer variety associated
to a class in $H^1(\mathcal{G}, \mathrm{PGL}_2)$. Let
$\mathfrak{A}$ be the algebra defined by the corresponding class
in $H^1(\mathcal{G}, \mathrm{Au})$. 
For each intermediate extension $K \subseteq F \subseteq L$, 
the set $\mathbb{V}(F)$ 
can be interpreted as the set of homothety classes of
nilpotent elements in $\mathfrak{A}_F = \mathfrak{A} \otimes_K F$.
\end{Example}

In what follows, we refer to the points of $\mathrm{Triv}$
as the trivializations of the algebra.
Note that $\mathcal{G}$ acts on the set 
$\mathrm{Triv}(L)$ in a way that 
$\sigma\big(f(\mathtt{q})\big)=
{}^\sigma f\big(\sigma(\mathtt{q})\big)$, 
for all $\mathtt{q}\in \mathfrak{A}_L$.
Equivalently, the action on $\mathrm{Triv}(L)$ 
is defined by
${}^\sigma f(\mathtt{q})=
\sigma\Big(f\big(\sigma^{-1}(\mathtt{q})\big)\Big)$.
Fix a trivialization $f\in \mathrm{Triv}(L)$, i.e.,
$f: \mathfrak{A}_L \to \mathbb{M}_2(L)$ is an isomorphism.
It is straightforward that the map
$a: \mathcal{G} \to \mathrm{Au}(L)$, 
 defined by $\sigma \mapsto f\circ \left( {}^{\sigma} f \right)^{-1}$,
 is a $1$-cocycle. Since $\mathrm{PGL}_2(L)$ embeds in the group
$\mathrm{Simp}(\mathcal{T}_L)$, the map $a$ can be regarded as a
$1$-cocycle with values in the latter.
In particular, its class $\bar{a}=[a]$ defines an element in 
$ H^1(\mathcal{G}, \mathrm{Simp}(\mathcal{T}_L))$.
If we choose another trivialization $f': \mathfrak{A}_L \to \mathbb{M}_2(L)$, 
then the associated cocycle 
$a': \mathcal{G} \to \mathrm{Au}(L)$ 
is cohomologous to $a$.
Thus, $[a]=[a']$ as an element in 
$H^1(\mathcal{G}, \mathrm{Au})\cong H^1(\mathcal{G}, \mathrm{PGL}_2)$.
The same holds, therefore, for its image in
$H^1\big(\mathcal{G}, \mathrm{Simp}(\mathcal{T}_L)\big)$.

\begin{Lemma}\label{lemma injective arrow for forms}
Assume that $K$ is a local field and let $L/K$ be a even degree 
finite extension. Then, the image of the map 
$\theta : \mathrm{SB}_1(L/K)\to \mathcal{E}(L/K)$ described in the 
preceding paragraph is $\left\lbrace [\mathcal{T}_L], 
[\widehat{\mathcal{T}}_{L}] \right\rbrace$,
where $\widehat{\mathcal{T}}_L$ is a $L/K$-form of 
$\mathcal{T}_L$ without $\mathcal{G}$-invariant visual limits,
which corresponds to the unique quaternion division algebra defined 
over $K$. In particular $\theta$ is injective. 
\end{Lemma}

\begin{proof}
Recall that either, the set of isomorphism classes of quaternion $K$-algebras splitting at $L$ and $\mathrm{SB}_1(L/K)$ are in a pointed bijection with
the cohomology set $H^1(\mathcal{G}, \mathrm{PGL}_2)$.
Note that, when $K$ is a local field,
there exists precisely two non-isomorphic quaternion $K$-algebras
and both are split over $L$, since $[L:K] \in 2\mathbb{Z}$
\cite[\S XIII, Cor. 3]{Serre-localfields}.
We conclude that $\mathrm{SB}_1(L/K)$ has precisely two elements.
It suffices to observe that, for the non-trivial element,
the associated $L/K$-form $\widehat{\mathcal{T}}_L$ has no 
$\mathcal{G}$-invariant visual limit according to Lemma 
\ref{lemma visual limits as severi-brauer varieties},
whence $\widehat{\mathcal{T}}_L$ and $\mathcal{T}_L$ are not isomorphic.
\end{proof}

\begin{Example}
 Lemma \ref{lemma injective arrow for forms} actually holds  
 for any discretely valued field $K$, as long as
 the $2$-torsion $\mathrm{Br}(K)_2$ of the Brauer group 
 $\mathrm{Br}(K)$ is isomorphic to $\mathbb{Z}/2\mathbb{Z}$.
It follows from the exact sequence:
$$ 0 \to \mathrm{Br}(\mathbb{K}) \to \mathrm{Br}(K) \to H^1\big(\mathrm{Gal}
(\bar{\mathbb{K}}/\mathbb{K}), \mathbb{Q}/\mathbb{Z}\big) \to 0,$$
that the condition on $\mathrm{Br}(K)$ holds, for instance,
when $\mathrm{Br}(\mathbb{K})_2=\lbrace 0\rbrace$ and 
$\mathrm{Gal}(\bar{\mathbb{K}}/\mathbb{K})\cong \widehat{\mathbb{Z}}$.
This is the case of $K=k((t))((x))$ with the $x$-adic valuation, when 
$\mathrm{char}(k)=0$ and $\overline{k}=k$, according to \cite[Cor. of 
Prop. 8, \S IV.2]{Serre-localfields}.
\end{Example}

In the sequel, we let $L/K$ be an arbitrary finite 
extension of local fields,
and we write $\mathcal{G}_E:=\mathrm{Gal}(E/K)$, for each 
intermediate Galois extension $K \subseteq E \subseteq L$.
Since $\mathcal{G}_{L/E}:=\mathrm{Gal}(L/E)$ is 
normal in $\mathcal{G}_L$ and
$\mathcal{G}_E \cong \mathcal{G}_L/\mathcal{G}_{L/E}$
via the restriction map,
for any algebraic group $G$ there exists an inflation map
$$\mathrm{Inf}: H^1(\mathcal{G}_E,G) \to H^1(\mathcal{G}_L,G).$$
However, we cannot define an inflation map from
$H^1\big(\mathcal{G}_E, \mathrm{Simp}(\mathcal{T}_E)\big)$ to
$H^1\big(\mathcal{G}_L, \mathrm{Simp}(\mathcal{T}_L)\big)$,
since we have no natural way to identify
 $\mathrm{Simp}(\mathcal{T}_E)$ with a subset of
  $\mathrm{Simp}(\mathcal{T}_L)$. In spite of that, since 
  the map $\theta$ is injective, we can still define an 
  inflation map that associates an
  $L/K$-form to every $E/K$-form defined from a Severi-Brauer variety,
  by passing through the set $H^1(\mathcal{G}_E,\mathrm{Au})\cong 
  H^1(\mathcal{G}_E,\mathrm{PGL}_2)$.

  When $\mathcal{T}'$ is a $E/K$-form such that 
  $\left[\mathcal{T}' \right]=\theta([\mathbb{V}])$ 
  we can assume that the Galois action in 
  $\mathcal{T}'$ is given by 
  $\tau * x= a_{\tau}\big(\tau(x)\big)$, where $a_{\tau}$
  is a simplicial map on $\mathcal{T}_E$
  , defined by a Moebius transformation,
  and $\tau\mapsto a_\tau$ is a cocycle.
  As such, we can define a cocycle 
  $(\tilde{a}_s)_{s \in \mathcal{G}_L}\in
  Z^1\big(\mathcal{G}_L, \mathrm{Simp}(\mathcal{T}_L)\big)$,
  where $\tilde{a}_s$ is the simplicial map 
  on $\mathcal{T}_L$ defined
  by the same Moebius transformation as $a_{\bar{s}}$,
  where $\bar{s}\in\mathcal{G}_E $
  is the restriction of $s\in\mathcal{G}_L$.
  If we assume $x\in \mathcal{T}_E\subseteq\mathcal{T}_L$, 
  so that $\tau(x)$ depends only
  on the class $\bar{\tau}$ of $\tau$ in $\mathcal{G}_E$, 
  the element $\bar{\tau}*x\in T_E$ coincides with
  $\tau*x\in T_L$. We can, therefore, consider $\mathcal{T}'$
  as a subset of the corresponding $L/K$-form, which we denote 
  $\widetilde{\mathrm{Inf}}\left(\mathcal{T}'\right)$.
  Furthermore, since $a_{\bar{s}}$ is trivial for
  $s\in\mathcal{G}_{L/E}$, then the action of
  $\mathcal{G}_{L/E}$ on 
  $\widetilde{\mathrm{Inf}}\left(\mathcal{T}\right)$
  coincides with the classical action, i.e., $\mathcal{T}'$
  can be recovered as the minimal sub-tree of
  $\widetilde{\mathrm{Inf}}\left(\mathcal{T}\right)$ containing 
  every $\mathcal{G}_{L/E}$-invariant visual limit.
  Now, assume $E/K$ has even degree, and let 
  $\widehat{\mathcal{T}}=\widehat{\mathcal{T}}_E$ be the form 
  defined by the unique division algebra, as described in 
  Lemma \ref{lemma injective arrow for forms}. Since 
  $\widehat{\mathcal{T}}_E$ has no invariant visual limit, then 
  neither do $\widehat{\mathcal{T}}_L$ for any Galois extension 
  $L/K$ with $E\subseteq L$. We conclude that
  $\widetilde{\mathrm{Inf}}\big(\widehat{\mathcal{T}}_E\big)$
  is isomorphic to $\widehat{\mathcal{T}}_L$. In particular, 
  next result is immediate: 

\begin{Lemma}\label{lemma embb of non-trivial forms}
Let $K$ be a local field, and consider even degree
Galois extensions $E/K$ and $L/K$ with $ E \subseteq L$.
Let $\widehat{\mathcal{T}}_L$ be as in Lemma 
\ref{lemma injective arrow for forms}. Then the space $\widehat{\mathcal{T}}_E$ is naturally
isomorphic to the  minimal subtree of $\widehat{\mathcal{T}}_L$ 
containing every $\mathcal{G}_{L/E}$-invariant visual limit. \qed
\end{Lemma}

\subparagraph{Forms of $\mathcal{T}_L$ not defined from 
quaternion algebras:}\label{subsection forms and quad alg}

It is rather intuitive that the map 
$\theta$ defined in Lemma \ref{lemma injective arrow for forms} fails 
to be surjective. It is relatively easy to construct forms
$\mathcal{T}'$ of $\mathcal{T}_L$ whose set of invariant visual 
limits is finite. The corresponding isomorphism classes
$[\mathcal{T}']$ are not 
defined by Severy-Brauer varieties in $\mathrm{SB}_1(L/K)$, 
via $\theta$, according to Lemma 
\ref{lemma injective arrow for forms}. 
Here we present two examples of these constructions.

\begin{Example}\label{ex unramified}
Let $K$ be a non-dyadic local field and let $L/K$ be an unramified 
quadratic extension. Let us write 
$\mathcal{G}=\mathrm{Gal}(L/K)=\lbrace \mathrm{id}, \sigma \rbrace$,
which identifies canonically with the Galois group of the
corresponding extension $\mathbb{L}/\mathbb{K}$ of residual fields.
Since the cardinality $|\mathbb{K}^{*}|$ is even, there exists a
involution map $h: \mathbb{K}^{*} \to \mathbb{K}^{*}$ 
without fixed points. Let $v_n=B_{0,n}$ be an arbitrary vertex 
of $\mathcal{P}_L(0,\infty)$. The $1$-star $\mathcal{S}_n$ of 
$v_n$ is the minimal subtree of $\mathcal{T}_L$ containing $v_n$ 
and all its neighbors. Note that the vertex set $V(\mathcal{S}_n) 
\smallsetminus\{v_n\}$ is in bijection with 
$\mathbb{P}^1(\mathbb{L})= \mathbb{K}^*\cup (\mathbb{L} 
\smallsetminus \mathbb{K}) \cup \lbrace 0,\infty \rbrace$,
with the element $\lambda \in \mathbb{P}^1(\mathbb{L})$ 
corresponding to 
a vertex $w_\lambda$.
Hence, there exists a simplicial map 
$f_n: \mathcal{S}_n \to \mathcal{S}_n$ such that:
\begin{itemize}
    \item $f_n $ fixes point-wisely the vertex set corresponding to 
    $\mathbb{L} \smallsetminus \mathbb{K} \cup 
    \lbrace 0,\infty \rbrace$,
    \item $f_n $ moves every vertex corresponding to an element
    in $\mathbb{K}^*$, in a way that 
    $f_n(w_\lambda)=w_{h(\lambda)}$.
\end{itemize}
Let $\mathfrak{c}$ be the caterpillar subtree containing precisely the 
vertices at distance $1$ from $\mathcal{P}_L(0,\infty)$,
let $C\subseteq T_L$ be its associated topological space,
and $\mathcal{C}=(\mathfrak{c},C)$ the associated topological graph.
Since $f_n$ fixes the vertices $w_{0}$ and $w_{\infty}$, there exists 
a simplicial map $f_\mathcal{C}: \mathcal{C} \to \mathcal{C}$ such that
$f_\mathcal{C}\lvert_{\mathcal{S}_n}=f_n$, for all $n\in \mathbb{Z}$. 
Since $f_n^2=\mathrm{id}$ for all $n$, the map $f_\mathcal{C}$ is 
also an involution. Furthermore since $f_n$ fixes $w_\lambda$, for 
$\lambda \in \mathbb{L} \smallsetminus \mathbb{K}$, 
we have ${}^{\sigma} f_n=f_n$, whence 
${}^{\sigma} f_\mathcal{C}=f_\mathcal{C}$.
For each leaf $w$ in the caterpillar, we let $\mathcal{T}_w$
be the corresponding connected component of the graph obtained
by removing $\mathcal{P}_L(0,\infty)$ and its adjacent edges.
To extend $f_\mathcal{C}$ to a $\mathcal{G}$-invariant involution 
$f:\mathcal{T}_L\rightarrow \mathcal{T}_L$, we can define it as 
the identity in every tree $\mathcal{T}_w$ for 
$w=w_\lambda\in V(\mathcal{S}_n)$
satisfying $\lambda\in \mathbb{L}\smallsetminus\mathbb{K}$,
and choosing a $\mathcal{G}$-invariant simplicial bijection between 
the trees $\mathcal{T}_w$ and $\mathcal{T}_{w'}$ for every pair
of leaves $w$ and $w'$ that $f_\mathcal{C}$ permutes.
Thus defined, $f$ induces a $1$-cocycle in 
$\mathrm{Simp}(\mathcal{T}_L)$ such that $f_{\sigma}=f$. 
Note that this is possible since $f^2=\mathrm{id}$, which is consistent
with the condition $f_1=f_{\sigma^2}= f_\sigma{}^\sigma(f_\sigma)$.
In particular,  this cocycle corresponds to a 
non-trivial form $\mathcal{T}'$ of $\mathcal{T}_L$.
Moreover, the action of $\mathcal{G}$ on $\mathcal{T}'$
is given by the twisted formula $\tau *x:= f_{\tau}(\tau x)=f(\tau x)$,
for all $\tau \in \mathcal{G}$ and all $x\in \mathcal{T}_L$, 
according to \cite[\S I.5.3]{Serre-CohomologieGaloisienne}.
Note that the $K$-ghost neighboring vertices of $v_n$ are fixed by $f$, 
but permuted by $\mathcal{G}$, 
while that the neighboring $K$-vertices of $v_n$, 
different from $w_0, w_\infty$, are permuted by $f$ and trivially
fixed by $\mathcal{G}$.
Thus, the unique $\mathcal{G}$-fixed visual limits of 
$\mathcal{T}'$ are $\lbrace 0, \infty \rbrace$.
\end{Example}

\begin{Example}\label{ex ramified}
Let $K$ a non-dyadic local field and let $L/K$ be a ramified 
quadratic extension.
Here we keep and extend the notations of Ex. \ref{ex unramified}.
In particular, for each vertex $v_{n/2}=B_{0,n/2}$ in 
$\mathcal{T}_L$ we denote by $\mathcal{S}_{n/2}$ its $1$-star.
Since $V(\mathcal{S}_{n/2}) \smallsetminus \lbrace v_{n/2} \rbrace$ is
in bijection with $\mathbb{P}^1(\mathbb{L})=\mathbb{P}^1(\mathbb{K})=
\mathbb{K}^*\cup \lbrace 0, \infty \rbrace$, there exist simplicial
maps $f_n:\mathcal{S}_{n/2} \to \mathcal{S}_{n/2} $ such that:
\begin{itemize}
   \item for $n \in 2\mathbb{Z}+1$, $f_n $ fixes point-wisely 
   $\mathbb{P}^1(\mathbb{K})$, and
   \item for $n\in 2\mathbb{Z}$, $f_n$ fixes point-wisely $w_0$ and
   $w_\infty$, while it moves every vertex corresponding to an element
    in $\mathbb{K}^*$, in a way that 
    $(f_n)_V(w_\lambda)=w_{h(\lambda)}$, where $h$ is a above.
\end{itemize}
Thus, by reasoning as in Ex. \ref{ex unramified}, we can glue the 
preceding local simplicial maps, in order to obtain a constant 
$1$-cocycle $f:\mathcal{G} \to \mathrm{Simp}(\mathcal{T}_L)$.
Again, this defines an $L/K$-form $\mathcal{T}'$ whose invariant 
visual limit is $\lbrace 0,\infty \rbrace$. 
\end{Example}

\section{Branches in twisted forms of Bruhat-Tits trees}\label{section branches in galois forms}

In all of this section, we fix a Galois extension $L/K$ of local fields
with Galois group $\mathcal{G}=\mathrm{Gal}(L/K)$ and we let
$\mathfrak{A}$ be the unique quaternion algebra 
defined over $K$. As before, we assume that $\mathfrak{A}$ 
splits over $L$, or, equivalently, we assume that the order 
$|\mathcal{G}|$ is even \cite[\S XIII, Cor.~3]{Serre-localfields}.
Let $[a]\in H^1(\mathcal{G},\mathrm{Au})$ be the class 
corresponding to $\mathfrak{A}$, whose image in 
$H^1\big(\mathcal{G},\mathrm{Simp}(\mathcal{T}_L)\big)$ 
corresponds to the $L/K$-form $\widehat{\mathcal{T}}_L$
described in the preceding section.

Let $\mathfrak{H}$ be an arbitrary quaternion order in $\mathfrak{A}$.
In the sequel, we identify $\mathfrak{H}$ with its image 
$\mathfrak{H} \otimes_{\mathcal{O}_K} \mathcal{O}_K
\subseteq\mathfrak{A}_L$. 
In particular, an $\mathcal{O}_L$-order
in $\mathfrak{A}_L$ contains $\mathfrak{H}$ if and only if
it contains $\mathfrak{H}_L=\mathcal{O}_L\mathfrak{H}$.
Furthermore, if we fix a trivialization
$f: \mathfrak{A}_L\to \mathbb{M}_2(L)$, then we also identify
$\mathfrak{H}$ with its image $f(\mathfrak{H})\subseteq\mathbb{M}_2(L)$. 
The map $a_{\tau}=f \circ ({}^{\tau}f )^{-1}\in 
\mathrm{Au}(L)$, defines a representative
of the class $[a]$ corresponding to $\mathfrak{A}$, as above.
The correspondence between cohomology classes and algebras
is such that the Galois action on $\mathfrak{A}_L$ can be defined
by the formula $f\big(\tau(\mathtt{q})\big)=
a_\tau \Big(\tau\big(f(\mathtt{q})\big)\Big)$,
for every $\mathtt{q}\in \mathfrak{A}_L$.
In particular, if we associate to every maximal order $\mathfrak{D}$
in $\mathfrak{A}_L$ the vertex 
$v_{\mathfrak{D}}:=v_{f(\mathfrak{D})}$,
then this association is consistent with the twisted 
Galois action in 
the following sense:
\begin{equation}\label{aconsp}
\tau* v_{\mathfrak{D}}=v_{\tau(\mathfrak{D})}.
\end{equation}
If $f$ is replaced by a different trivialization $g=h\circ f$, 
where $h\in \mathrm{Au}(L)$, i.e., $h$ is an automorphism of 
the matrix algebra defined over $L$, the vertex
$v_{f(\mathfrak{D})}$ is replaced with
$v_{g(\mathfrak{D})}=\gamma \bullet v_{f(\mathfrak{D})}$,
where $\gamma$ is the Moebius transformation corresponding
to $h$, but the structure as a simplicial complex remains invariant.
Furthermore, the relation (\ref{aconsp}), tells us that
the $\mathcal{G}$-set structure is also independent of this choice,
in the sense that $\tau\star v_{g(\mathfrak{D})}=
\gamma\bullet \left(\tau*v_{g(\mathfrak{D})}\right)$,
when $\star$ denotes the twisted action corresponding to $g$.
With these identifications, for
any maximal order $\mathfrak{D}\subseteq\mathfrak{A}_L$,
the vertex $v_{\mathfrak{D}}$
is well defined as a vertex of the twisted tree
$\widehat{\mathcal{T}}_L$.

For any order $\mathfrak{H}\subseteq\mathfrak{A}$, we define its branch
by $\mathcal{S}_{L,f}(\mathfrak{H})=
\mathcal{S}_L\big(f(\mathfrak{H}_L)\big)$. This depends on the 
trivialization $f$, but it is well defined as a subgraph
of the topological tree $\widehat{\mathcal{T}}_L$, in the sense
described above. For this reason, we often omit $f$ from the notation 
and write simply $\mathcal{S}_L(\mathfrak{H})$. Similar conventions
apply to individual quaternions. The branch denoted by
$\mathcal{S}_{L,f}(\mathtt{q})$ or $\mathcal{S}_L(\mathtt{q})$ 
is the branch $\mathcal{S}_L(\mathbf{q})$ where
$\mathbf{q}=f(\mathtt{q})$.

\begin{Lemma}
For each order $\mathfrak{H} \subset \mathfrak{A}$, we have 
$\tau* \mathcal{S}_L(\mathfrak{H}) = \mathcal{S}_L(\mathfrak{H})$ 
for every $\tau\in\mathcal{G}$.
\end{Lemma}

\begin{proof}
Since $\mathcal{G}$ acts trivially on $\mathfrak{H}$, then, for every
maximal order $\mathfrak{D}$ in $\mathfrak{A}_L$ 
containing $\mathfrak{H}$
the order $\tau(\mathfrak{D})$ also contains $\mathfrak{H}$.
This implies that $\tau*v_{\mathfrak{D}} \in
\mathcal{S}_L(\mathfrak{H})$,  for every vertex 
$v_{\mathfrak{D}} \in \mathcal{S}_L(\mathfrak{H})$.
Since the $\mathcal{G}$-action on $\mathcal{T}_L$ is simplicial, 
the result follows.
\end{proof}

If $S_L(\mathfrak{H})$ is the topological space subjacent
to the branch $\mathcal{S}_L(\mathfrak{H})$,
we denote by $S_L^{\mathcal{G},*}(\mathfrak{H})$ the (point-wise)
$\mathcal{G}$-invariant subspace of $S_L(\mathfrak{H})$, which we call
in the sequel $\mathcal{G}$-invariant branch (or just the 
$\mathcal{G}$-branch) of $\mathfrak{H}$.
When $\mathfrak{H}=
\mathcal{O}_K[\mathtt{q}_1, \cdots ,\mathtt{q}_r]$, we write
$S_L^{\mathcal{G},*}(\mathtt{q}_1, \cdots, \mathtt{q}_r)$ instead of 
$S_L^{\mathcal{G},*}(\mathfrak{H})$. Note that these
invariant subspaces could, in principle, fail to define subgraphs.

Recall that, for every even degree
intermediate Galois 
extension $E/K$, we can identify $\widehat{\mathcal{T}}_E$ with
the minimal subtree containing the $\mathcal{G}_{L/E}$-invariant
visual limits, according to Lemma \ref{lemma embb of non-trivial forms}. 
If $f$ is a trivialization defined over $E$, for any maximal
order $\mathfrak{D}'$ in $\mathfrak{A}_E$ we identify
the vertex $v_{\mathfrak{D}'}$ in the tree $\widehat{\mathcal{T}}_E$
with the vertex $v_{\mathfrak{D}}$ in $\widehat{\mathcal{T}}_L$ corresponding
to the maximal order
$\mathfrak{D}=\mathfrak{D}'\otimes_{\mathcal{O}_E}\mathcal{O}_L$.
This is consistent with the previous definitions
and gives an interpretation of the embedding of $\widehat{\mathcal{T}}_E$
into $\widehat{\mathcal{T}}_L$ that is independent of the trivialization.
This is important since
 we might fail to have a trivialization
$f$ that is defined over every intermediate Galois extension
of even degree, since this condition
forces the trivialization to be defined over $K$,
since this is the intersection of the unramified
and ramified quadratic extensions.
By identifying $S_E(\mathfrak{H})$ (resp. 
$S_E^{\mathcal{G}_E,*}(\mathfrak{H})$) with its image via the map
just described, next result is straightforward:

\begin{Lemma}\label{lemma galois cont}
Under the preceding conventions, we have both
$S_E(\mathfrak{H}) \subseteq S_L(\mathfrak{H})$ and 
$S_E^{\mathcal{G}_E,*}(\mathfrak{H})\subseteq 
S_L^{\mathcal{G}_L,*}(\mathfrak{H})$.\qed
\end{Lemma}

\begin{Example}\label{ex pelitos}
Let $\mathfrak{H}=\mathcal{O}_K[\mathrm{id}]$ and let $L/K$ be a 
wildly ramified extension.
By simplicity, assume that $E=K$ and $\mathfrak{A}=\mathbb{M}_2(K)$.
Then $S_E^{\mathcal{G}_E,*}(\mathfrak{H})= S_K(\mathfrak{H})=T_K$, 
while $S_L^{\mathcal{G}_L,*}(\mathfrak{H})=T_L^{\mathcal{G}_L}$.
Since $T_K$ is strictly contained in $T_K^{\mathcal{G}_L}$ according 
to \cite[\S2.6.1]{Titsreductivegroups}, \cite{Rousseau} or Prop. \ref{prop hairs}, 
the second contention in Lemma \ref{lemma galois cont} can be strict. 
\end{Example}

Having developed the theory up to this point, we have the tools 
to prove Ths. \ref{main teo 0} and \ref{main teo 0b}.

\begin{proof}[Proof of Th. \ref{main teo 0}]
Let us fix a trivialization $f: \mathfrak{A}_L \to \mathbb{M}_2(L)$, 
and assume $\widehat{\mathcal{T}}_L$ is defined by $f$.
Then, by Lemma \ref{lemma embb of non-trivial forms},
for every intermediate field $E$ splitting $\mathfrak{A}$,
we can identify the twisted form $\widehat{\mathcal{T}}_E$
with the minimal subtree of $\widehat{\mathcal{T}}_L$ containing
every $\mathcal{G}_{L/E}$-invariant visual limit. On the 
other hand, $\widehat{\mathcal{T}}_E$ is homeomorphic to the
Bruhat-Tits tree $\mathcal{T}_E$ over $E$, as it 
is a twisted form. Statement $(1)$ follows.
Next we focus on the proof of Statement $(2)$.
Indeed, we define $\mathcal{S}$ as the branch 
$\mathcal{S}_L(\mathfrak{H}_L)$ of $\mathfrak{H}_L$ in $\widehat{\mathcal{T}}_L$, as described above.
Note that, if $f(\mathfrak{H}_E) \subseteq \mathfrak{D}_E$,
then $f(\mathfrak{H}_L) \subseteq 
\mathfrak{D}_L:=\mathfrak{D}_E 
\otimes_{\mathcal{O}_E}\mathcal{O}_L$, whence 
$\mathfrak{D}_E$ defines a vertex in 
$\mathcal{S} \cap \widehat{\mathcal{T}}_E$.
Conversely, any vertex in $\mathcal{S} \cap \widehat{\mathcal{T}}_E$ 
corresponds to an $\mathcal{O}_E$-maximal order $\mathfrak{D}_E$ 
such that $f(\mathfrak{H}_L) \subseteq \mathfrak{D}_L$, whence
the result follows from the fact that $\mathfrak{D}_E$ is
the set of $\mathcal{G}_{L/E}$-invariant elements in 
$\mathfrak{D}_L$.
\end{proof}

\begin{proof}[Proof of Th. \ref{main teo 0b}]
    It follows from Th. \ref{main teo 0} that the maximal 
    orders in $\mathfrak{A}_E$ containing $\rho(G)$ are 
    in correspondence with the vertices of $\mathcal{S}_E=\mathcal{S} \cap \widehat{\mathcal{T}}_E$,
    as a subgraph of $\widehat{\mathcal{T}}_E$. On the other hand,
    it is immediate that $\mathfrak{L}_E^*$ acts trivially
    by conjugation on the $\mathcal{O}_E$-order
    $\mathfrak{H}_E$ generated by $\rho(G)$. Then
    it acts on the set of vertices of $\mathcal{S}_E$, as
    follows from Th. \ref{main teo 0}. Now the result is a 
    consequence of the fact that the whole 
    $\mathrm{GL}_2(E)$-action
    on the corresponding Bruhat-Tits tree is simplicial.
    Now, we can identify $\widehat{\mathcal{T}}_E$ with the
    Bruhat-Tit tree $\mathcal{T}_E$ by ignoring the Galois action,
    so the existence of the bijection follows from Prop. \ref{prop equiv rep and conj class}.
\end{proof}

\section{Integral representations of maximal orders in division algebras}\label{subsection int rep max orders}

Assume that $K$ is a local
field with $\mathrm{char}(K)\neq 2$.
The (unique) quaternion division algebra over $K$ can be presented as follows:
$$\mathfrak{A}=\Big(\frac{\pi_K,\Delta}{K}\Big)= 
K \left[ \mathtt{i},\mathtt{j} \vert 
\mathtt{i}\mathtt{j}+\mathtt{j}\mathtt{i}=\mathtt{0},  
\mathtt{i}^2=\pi_K\mathtt{1}, 
\mathtt{j}^2=\Delta\mathtt{1} \right],$$
where $\sqrt{\Delta}$ generates the unramified quadratic 
extension $F$ of $K$, while $\pi_K$ is one fix 
uniformizing parameter of $K$. Note that $\mathtt{0}$ and $\mathtt{1}$ denote quaternions, and we keep $\mathtt{1}$
in expressions of the form $a\mathtt{1}$, with
$a\in K$ to emphasize the fact that the product is
a quaternion. Let $L\subseteq\overline{K}$ be a Galois
extension of $K$ of even degree
containing both $\sqrt{\Delta}$ and $\sqrt{\pi_K}$.
Let $\tau\in\mathcal{G}=\mathrm{Gal}(L/K)$ be an element satisfying 
both $\tau(\sqrt\Delta)=-\sqrt\Delta$ and 
$\tau(\sqrt{\pi_K})=\sqrt{\pi_K}$. 
Let $\mathcal{H}=\mathrm{Gal}(L/F)$.
Note that $\tau^2\in\mathcal{H}$, but it need not
be trivial.

We choose the trivialization $f: \mathfrak{A}_{L} \to \mathbb{M}_2(L)$ 
given by:
\begin{equation}\label{eq trivialization}
\mathbf{i}=f(\mathtt{i})= \sbmattrix{0}{1}{\pi_K}{0}, 
\quad \mathbf{j}=f(\mathtt{j})= 
\sbmattrix{\sqrt{\Delta}}{0}{0}{-\sqrt{\Delta}}.
\end{equation}
Thus, the corresponding $1$-cocycle $a:\mathcal{G} 
\to \mathrm{PGL}_2(L)$ satisfies $a_{\tau}=\overline{\sbmattrix{0}{1}{\pi_K}{0}}$,
while $a_{\sigma}$ is trivial for $\sigma\in\mathcal{H}$, as the trivialization is 
defined over $F$. Accordingly, we assume that the twisted form $\widehat{\mathcal{T}}_L$ is the usual Bruhat-Tits tree
endowed with the twisted Galois action defined on the set
$\partial_{\infty}(\widehat{\mathcal{T}}_L)\cong \mathbb{P}^1(L)$ 
of visual limits by the following formulas:
\begin{equation}\label{eq twisted form}
\sigma*z=\sigma(z), \quad \tau*z=\frac{1}{\pi_K \cdot \tau(z)}, 
\quad \forall z\in \mathbb{P}^1(L), \forall \sigma\in\mathcal{H}.
\end{equation}
Let us denote by $v_{1/2}$ the vertex of $\widehat{\mathcal{T}}_L$
corresponding to the maximal order
$$\mathfrak{D}_{1/2}=\left\{\sbmattrix abcd\Big|\nu(a)\geq0, \nu(b)\geq\frac12,
\nu(c)\geq-\frac12, \nu(d)\geq0\right\}\subseteq\mathbb{M}_2(L)$$
via the trivialization $f$.
This vertex is not 
defined over $F$, so is the baricenter of an edge in 
$\widehat{\mathcal{T}}_{F}$. However, it is a vertex for
every subextension that ramifies over $K$.
Next result is now straightforward:

\begin{Lemma}\label{lemma action on p0inf}
The action of each $\mu\in \tau \mathcal{H}$, i.e., 
of each $\mu \in \mathcal{G}$ that satisfies 
$\mu(\sqrt{\Delta})=-\sqrt{\Delta}$, interchanges 
the visual limits $0$ and $\infty$, leaving the vertex 
$v_{1/2}$ as the only invariant vertex in the path 
$\mathcal{P}_L(0,\infty)$. \qed
\end{Lemma}

Now, we have the tools to prove Cor. \ref{main teo 1}.

\begin{proof}[Proof of Cor. \ref{main teo 1}]
The (unique) maximal order $\mathfrak{D}=
\mathfrak{D}_{\mathfrak{A}}$ of $\mathfrak{A}$ can be 
defined by the formula $\mathfrak{D}=\mathcal{O}_K\left[\mathtt{i}, 
\frac12(\mathtt{j}-\mathtt{1})\right]$.
We need to describe its representations into 
$\mathbb{M}_2(\mathcal{O}_E)$, for any intermediate extension 
$K \subseteq E \subseteq L$. In order to do this, start by 
computing the branch $\mathcal{S}_L(\mathfrak{D})$. Indeed, 
we choose the trivialization $f: \mathfrak{A}_{L} \to \mathbb{M}_2(L)$ 
described in Eq. \eqref{eq trivialization}. Thus, the twisted
form $\widehat{\mathcal{T}}_L$ is described by the formulas in 
Eq. \eqref{eq twisted form}. Since 
$\mathbf{w}:=\frac{1}{2\sqrt{\Delta}}\big(\sqrt{\Delta} 
\mathbf{1} -\mathbf{j}\big) = \sbmattrix{0}{0}{0}{1}$ is 
idempotent, we have that 
$\mathcal{O}_{L}[\mathbf{w}]\cong \mathcal{O}_{L} 
\times \mathcal{O}_{L}$. 
Thus, $S_L(\mathbf{w})$ is a maximal path according to 
Lemma \ref{splitbranches}.
By looking at the invariant visual limits of the Moebius 
transformation defined by $\mathbf{1}-2\mathbf{w}$, we 
get $\mathcal{S}_L(\mathbf{w})=\mathcal{P}_L(0,\infty)$.
Note that $\frac{\mathbf{j}-\mathbf{1}}{2}= 
\frac{\sqrt{\Delta}-1}{2} \mathbf{1}-\sqrt{\Delta} \mathbf{w}$,
whence $\mathcal{S}_L\left(\frac{\mathbf{j}-\mathbf{1}}{2}\right)= 
\mathcal{S}_L(\mathbf{w})= \mathcal{P}_L(0,\infty)$.
On the other hand, since $\mathbf{h}:=\frac{1}{2\sqrt{\pi_K}}\big(\sqrt{\pi_K} 
\mathbf{1} -\mathbf{j}\big)$ is 
idempotent, we have that $\mathcal{S}_L(\mathbf{h})$ is a maximal path.
Then, Lemma \ref{lemma contraidos} implies that 
$\mathcal{S}_{L}(\mathbf{i})=
\mathcal{P}_L(a,b) ^{[\nu(2\sqrt{\pi_K})]}$, 
for some pair  of distinct visual limits $a,b \in \mathbb{P}^1(L)$.
Since the irreducible polynomial of 
$\mathbf{q}=\frac{1}{\sqrt{\pi_K}}\mathbf{i}$ is $x^2-1$, 
the matrix $\mathbf{q}$ is 
integral over $\mathcal{O}_L$.
Thus, it follows from Lemma \ref{lemma branch and MT} 
that the visual limits $a$ and $b$ of the branch 
$\mathcal{S}_L(\mathbf{i})$ are the fixed points
of the Moebius transformation $\eta$ defined by 
$\eta(z)=\frac{1}{z\pi_K}$, which corresponds to
$\mathbf{i}$. In other words, we can assume 
that $a=\frac{1}{\sqrt{\pi_K}}$ and 
$b=-\frac{1}{\sqrt{\pi_K}}$. Thus, we have:
$$ \mathcal{S}_L(\mathbf{i})= 
\mathcal{P}_L\left( \frac{1}{\sqrt{\pi_K}}, 
-\frac{1}{\sqrt{\pi_K}}\right) ^{[\nu(2\sqrt{\pi_K})]}.$$
Therefore $\mathcal{S}_L(\mathfrak{D})=\mathcal{S}_L(\mathbf{i}) 
\cap \mathcal{S}_L\left(\frac{\mathbf{j}-\mathbf{1}}2\right)$ is the 
finite line $\mathcal{P}_L[v_0,v_1]$ joining the vertices 
$v_0$ and $v_1$, which, 
respectively, correspond to the 
$\mathcal{O}_L$-maximal orders 
$\mathfrak{D}_0=\mathbb{M}_2(\mathcal{O}_L)$ and
$\mathfrak{D}_1=\sbmattrix{\mathcal{O}_{L}}{\pi_K^{-1}
\mathcal{O}_{L}}{\pi_K\mathcal{O}_{L}}{\mathcal{O}_{L}}$.

Let $E$ be an intermediate extension $K \subseteq E \subseteq L$.
One the one hand, if $\sqrt{\Delta} \in E$, then 
the trivialization $f$ is defined over $E$, and therefore also
the finite line $\mathcal{P}_L[v_0,v_1]$.
Thus:
$$\mathcal{S}_L(\mathfrak{D})\cap \widehat{\mathcal{T}}_E= 
\mathcal{P}_L[v_0,v_1]\cap \widehat{\mathcal{T}}_E= \mathcal{P}_E[v_0,v_1].$$
On the other hand, if $\sqrt{\Delta} \notin E$, then there exists 
$\mu \in \mathcal{G}_{L/E}$ such that 
$\mu(\sqrt{\Delta})=-\sqrt{\Delta}$.
Then, it follows from Lemma \ref{lemma action on p0inf} that
$$S_E(\mathfrak{D})\subseteq
S_L(\mathfrak{D})^{\mathcal{G}_{L/E},*}= \lbrace v_{1/2}\rbrace,$$
whence Cor. \ref{main teo 1} follows.
\end{proof}

\begin{Example}\label{cor loc rep ij}
Note that, when $K$ is non-dyadic, we have that 
$\mathfrak{D}=\mathcal{O}_K\left[\mathtt{i}, 
\frac12(\mathtt{j}-\mathtt{1})\right] 
=\mathcal{O}_K\left[\mathtt{i}, \mathtt{j}\right]$.
So, under the hypotheses of Cor. \ref{main teo 1}, for each 
Galois extension $E$ of $K$ containing a quadratic extension, 
we have that:
\begin{itemize}
\item if $\sqrt{\Delta}\in E$, then there are exactly 
$e(E/K)+1$ conjugacy classes of representations of the form 
$\mathcal{O}_K\left[\mathtt{i}, \mathtt{j}\right] \hookrightarrow 
\mathbb{M}_2(\mathcal{O}_E)$, while
\item if $\sqrt{\Delta}\notin E$, all representations of the 
form $\mathcal{O}_K\left[\mathtt{i}, \mathtt{j}\right] 
\hookrightarrow \mathbb{M}_2(\mathcal{O}_E)$ are conjugates.
\end{itemize}
\end{Example}

Next, we consider the quaternion algebra
$\mathfrak{A}=\big(\frac{-1,-1}{\mathbb{Q}}\big)$, 
defined in terms of generator and relations by
$$\mathfrak{A}=\mathbb{Q}\big[ \mathtt{u},\mathtt{v} \vert \mathtt{u}^{2}=
\mathtt{v}^{2}=-\mathtt{1}, 
\mathtt{u}\mathtt{v}+\mathtt{v}\mathtt{u}=\mathtt{0} \big],$$
and the quaternion $\mathtt{w}=\frac{1}{2}(-\mathtt{1}+
\mathtt{u}+\mathtt{v}+\mathtt{u}\mathtt{v}) \in \mathfrak{A}$. 
The multiplicative group 
$G=\langle \mathtt{w}, \mathtt{u},\mathtt{v} \rangle
\subseteq \mathfrak{A}^*$ is called the Hurwitz unit group in current literature. 
This group has order $24$ and it is moreover isomorphic to 
$\mathrm{SL}_2(\mathbb{F}_3)$ according to 
\cite[Lemma 11.2.1]{Voight}. The subgroup generated by 
$\mathtt{u}$ and $\mathtt{v}$, which we denote by $Q_8$,
is called the quaternion group.

The group $\mathrm{SL}_2(\mathbb{F}_3)$ acts on the left on the size four set
$\mathbb{F}_3^*\backslash\mathbb{F}_3^2$ of nonzero column vectors 
up to sign, with $\lbrace \pm 1 \rbrace$ as kernel.
Moreover, since $A_4 \leq S_4$ is the unique subgroup of size $12$, the image of this (permutational) representation must be $A_4$. In other words, we have the exact sequence:
$$ 1 \to \lbrace \pm 1 \rbrace \to G\cong \mathrm{SL}_2(\mathbb{F}_3) \to A_4 \to 1.$$
The group $G$ is also isomorphic to the binary tetrahedral group, and it can be presented as:
$$ G\cong \langle r,s,t \vert r^2=s^3=t^3=rst=1 \rangle,$$
where $r=\mathtt{u}$, $s=\mathtt{w}$ and $t=\mathtt{w} \mathtt{v}^3$.
See \cite[\S 11.2.4]{Voight} for more details.

\begin{proof}[Proof of Th. \ref{main teo 2}]
Since $\mathbb{Z}[\mathtt{u},\mathtt{v},\mathtt{w}]$ is a
maximal $\mathbb{Z}$-order, we have that $\mathfrak{D}=\mathbb{Z}_2[\mathtt{u},\mathtt{v},\mathtt{w}]$ 
is the (unique) maximal $\mathbb{Z}_2$-order of the local 
division algebra $\mathfrak{A}_2:=
\mathfrak{A}\otimes_{\mathbb{Q}} \mathbb{Q}_2$.
Thus, the first statement of Th. \ref{main teo 2} follows from Cor. \ref{main teo 1}.
Now, let us consider $p\neq 2$, so that $\mathfrak{A}_p:=\mathfrak{A}\otimes_{\mathbb{Q}} \mathbb{Q}_p$ is a matrix algebra.
Since $\mathbb{Z}[\mathtt{u},\mathtt{v},\mathtt{w}]$ is a
maximal $\mathbb{Z}$-order, then
$\mathbb{Z}_p[\mathtt{u},\mathtt{v},\mathtt{w}]$ 
is a maximal $\mathbb{Z}_p$-order in $\mathfrak{A}_p$.
Thus, for each algebraic extension $E/\mathbb{Q}_p$ we have that $\mathcal{S}_E(\mathtt{u},\mathtt{v},\mathtt{w})$ contains just one vertex.
Now the second statement of Th. \ref{main teo 2} follows from Prop. 
\ref{prop equiv rep and conj class}.
\end{proof}

\begin{Example}\label{coro rep Q8 over p neq 2}
Let $E$ be an algebraic extension of $\mathbb{Q}_p$, for $p\neq 2$.
Let $\rho$ be the local quaternionic representation $\rho: Q_8 \to \mathrm{GL}_2(E)$ obtained by extending scalars to $E$.
Since at $p\neq 2$, the order $\mathbb{Z}_p[\mathtt{u},\mathtt{v},\mathtt{w}]$ equals $\mathbb{Z}_p[\mathtt{u},\mathtt{v}]$, which is the order
generated by $Q_8$,
all representations of the form $\rho': Q_8 \to \mathrm{GL}_2(\mathcal{O}_E)$ such that $\rho'\otimes_{\mathcal{O}_E} E=\rho$ are conjugates.
\end{Example}

Now, let us consider $\mathfrak{A}'=
\big(\frac{-3,-1}{\mathbb{Q}}\big)$, with generators 
$\mathtt{p}$ and $\mathtt{q}$
satisfying $\mathtt{p}^2=-\mathtt{1}$,
$\mathtt{q}^2=-\mathtt{3}$ and
$\mathtt{p}\mathtt{q}=-\mathtt{q}\mathtt{p}$,
and also the quaternion $\mathtt{r}=
\frac{1}{2}(1+\mathtt{q}) \in \mathfrak{A}'$.
The multiplictive 
group $G=\langle \mathtt{r}, \mathtt{p} \rangle$ is the 
dicyclic group, and it corresponds to the group of units of 
the $\mathbb{Z}$-order 
$\mathbb{Z}[\mathtt{r},\mathtt{p}]$.
The group $G$ has order $12$ and it is isomorphic to the only non-direct
semi-direct product $C_3 \rtimes C_4$.

\begin{proof}[Proof of Th. \ref{main teo 3}]
It follows from \cite[\S11.5.11(ii)]{Voight} that the $\mathbb{Z}$-order 
$\mathbb{Z}[\mathtt{r},\mathtt{p}]$, defined above, is  maximal.
Therefore $\mathfrak{D}'=\mathbb{Z}_3[\mathtt{r},\mathtt{p}]$ 
is the (unique)  maximal $\mathbb{Z}_3$-order of 
$\mathfrak{A}'_3=\mathfrak{A}'\otimes_{\mathbb{Q}} \mathbb{Q}_3$,
where $\mathfrak{A}'$ is as above.
Hence, the fist statement of Th. \ref{main teo 3} follows from Cor. \ref{main teo 1}.
Now, let us consider $p\neq 3$, so that $\mathfrak{A}'_p:=\mathfrak{A'}\otimes_{\mathbb{Q}} \mathbb{Q}_p$ is isomorphic to $\mathbb{M}_2(\mathbb{Q}_p)$.
Since $\mathbb{Z}[\mathtt{r}, \mathtt{p}]$ is a
maximal $\mathbb{Z}$-order, we have that $\mathbb{Z}_p[\mathtt{r}, \mathtt{p}]$ 
is a maximal $\mathbb{Z}_p$-order in $\mathfrak{A}'_p$.
Thus, for each algebraic extension $E/\mathbb{Q}_p$, the branch 
$\mathcal{S}_E(\mathtt{r}, \mathtt{p})$ consists in just one vertex.
Therefore, the second statement of Th. \ref{main teo 3} follows from Prop. \ref{prop equiv rep and conj class}.
\end{proof}

\section{Local integral representations of \texorpdfstring{$Q_8$}{Q8}}\label{section rep ham}

As we discuss in Cor. \ref{coro rep Q8 over p neq 2}, the local
integral representation of $Q_8$ over $\mathbb{Q}_p$, for $p \neq 2$,
can be studied by using the classical theory of branches in 
Bruhat-Tits trees as in Prop. \ref{prop equiv rep and conj class},
with no need of twistings, as the standard quaternion
algebra split at every odd prime.
Furthermore, there is a 
unique integral representation for such a 
group according to Cor. \ref{coro rep Q8 over p neq 2}.
For this reason, we assume $p=2$ throughout.

Recall that $K$ is a local field and $\rho: G \to \mathrm{GL}_n(K)$ is a representation.
For any extension of local fields $K/K'$, we denote by 
$\mathrm{IF}_{\rho}^G(K/K')$ 
the set of integral representations $\rho'$ over
$\mathcal{O}_K$ that are defined over $\mathcal{O}_{K'}$, 
and satisfy $\rho' \otimes_{\mathcal{O}_K} K \cong \rho$.
In the sequel, we focus into describing both $\mathrm{IF}_{\rho}^{Q_8}(K)$ and
$\mathrm{IF}_{\rho}^{Q_8}(K/K')$, 
for suitable field extensions $K/K'$ and for the local 
representation $\rho$ introduced in Ex. \ref{coro rep Q8 over p neq 2}.

Let 
$\mathfrak{A}=\big(\frac{-1,-1}{\mathbb{Q}}\big)=
\mathbb{Q}\big[ \mathtt{u}, \mathtt{v}]$, as at the end of
\S\ref{subsection int rep max orders},
and set $\mathfrak{H}=\mathbb{Z}[\mathtt{u}, \mathtt{v}] 
\subset \mathfrak{A}$.
The group of unit $\mathfrak{H}^{*}$ of $\mathfrak{H}$ is
isomorphic to the quaternion group $Q_8$ with the standard 
presentation, namely:
$$Q_8=\langle \mathtt{u}, \mathtt{v} \vert \mathtt{u}^4=\mathtt{1},
\mathtt{u}^2=\mathtt{v}^2, 
\mathtt{v}\mathtt{u}=\mathtt{u}^{-1}\mathtt{v} \rangle.$$
At the same time, consider $\mathfrak{A}''=\big(\frac{-2,-3}
{\mathbb{Q}}\big)=\mathbb{Q} \big[ \mathtt{i},\mathtt{j} \vert 
\mathtt{i}^2=-\mathtt{2}, \mathtt{j}^2= -\mathtt{3},  \mathtt{i} 
\mathtt{j}+\mathtt{j} \mathtt{i}=\mathtt{0}\big]$,
and let us consider the isomorphism $\phi: \mathfrak{A} \to 
\mathfrak{A}''$ defined by $\phi(\mathtt{u})=\frac{1}{6}(2 \mathtt{j}-
\mathtt{i}\mathtt{j}+3\mathtt{i})$ and $\phi(\mathtt{v})=\frac{1}{6}(2 
\mathtt{j}-\mathtt{i}\mathtt{j}-3\mathtt{i})$.
Let $L$ be a field extension of $\mathbb{Q}_2$ containing $\sqrt{-3}$.
The quaternion algebra $\mathfrak{A}''_L$ has the trivialization $f: 
\mathfrak{A}''_L \to \mathbb{M}_2(L)$ given by:
$$\mathbf{i}=f(\mathtt{i})=\sbmattrix{0}{1}{-2}{0}, \quad 
\mathbf{j}=f(\mathtt{j})=\sbmattrix{\sqrt{-3}}{0}{0}{-\sqrt{-3}}.$$
In particular, this defines a trivialization
$g=f\circ\phi:\mathfrak{A}_L\rightarrow\mathbb{M}_2(L)$ through
the following matrices:
\begin{equation*}
   \mathbf{u}:= g(\mathtt{u})= \frac{1}{6}
   \sbmattrix{2\sqrt{-3}}{\sqrt{-3}+3}{2\sqrt{-3}-6}{-2\sqrt{-3}}, 
   \quad \mathbf{v}:=g(\mathtt{v})= \frac{1}{6}
   \sbmattrix{2\sqrt{-3}}{\sqrt{-3}-3}{2\sqrt{-3}+6}{-2\sqrt{-3}}.
\end{equation*}
Let us write $\omega=\frac{-1+\sqrt{-3}}{2}$, 
so that the following identities hold:
\begin{equation}
    \mathbf{u}= \frac{1}{\sqrt{-3}}
    \sbmattrix{-1}{\omega}{-2(\omega+1)}{1}, 
    \quad \mathbf{v}= \frac{1}{\sqrt{-3}}
    \sbmattrix{-1}{-(\omega+1)}{2\omega}{1}.
\end{equation}
Assume that $\sqrt{-1} \in L$. Next we determine the branch 
$\mathcal{S}_L(\mathbf{u}, \mathbf{v})=
\mathcal{S}_L(\mathbf{u}) \cap \mathcal{S}_L(\mathbf{v})$.
Indeed, since the eigenvalues $\lambda_1=\sqrt{-1}$ 
and $\lambda_2=-\sqrt{-1}$ of $\mathbf{u}$
satisfy $\nu(\lambda_2-\lambda_1)=\nu(2)$, we have
$\mathcal{S}_L(\mathbf{u})=
\mathcal{P}_L(z_1,z_2)^{[\nu(2)]}$ for
suitable infinite limits $z_1$ and $z_2$.
See Ex. \ref{exa47}.
In fact, $z_1$ and $z_2$
are the fixed points of the Moebius transformation defined by 
$\mathbf{u}$, namely $z \mapsto \frac{-z+\omega}{-2(\omega+1)z+1}$.
Thus, $z_1=\frac{1+\sqrt{3}}{2(\omega+1)}$ and 
$z_2=\frac{1-\sqrt{3}}{2(\omega+1)}$.
Analogously $\mathcal{S}_L(\mathbf{v})=
\mathcal{P}_L(z_3,z_4)^{[\nu(2)]}$, where 
$z_3=\frac{-1+\sqrt{3}}{2\omega}$ and 
$z_4=\frac{-1-\sqrt{3}}{2\omega}$. Note that
$z_1z_3=z_2z_4=-\frac12$.
Since $\nu(z_1-z_2)=0$, the peak of the path 
$\mathcal{P}_L(z_1,z_2)$ is the vertex
$v_{1,2}:=B_{z_1, 0}=B_{z_2, 0}$.
Analogously, the peak of $\mathcal{P}_L(z_3,z_4)$ is 
$v_{3,4}:=B_{z_3,0}=B_{z_4,0}$.
Moreover, since $z_1-z_3=\frac{\sqrt{3}(1-\sqrt{-1})}{2}$, 
with the numerator having valuation $\nu(2)/2$,
the peak of $p(z_1,z_3)$ is $v_c:=B_{z_1,-\nu(2)/2}$.

\begin{figure}
\[ 
\fbox{
\xygraph{!{<0cm,0cm>;<0.9cm,0cm>:<0cm,0.9cm>::}
!{(6,5.8) }*+{\textbf{(A)}}="A"
!{(0,3.3) }*+{\circ}="b"
!{(3,3.3) }*+{\circ}="a"
!{(-0.6,3.3) }*+{\hphantom{x}^{v_{1,2}}}="bim"
!{(3.4,3.3) }*+{\hphantom{x}^{v_{3,4}}}="aim"
!{(1.5,4.3) }*+{\ast}="f"
!{(1.5,4.6) }*+{\hphantom{x}^{v_c}}="fim"
!{(0.5,4.5) }*+{\bullet}="x"
!{(2.7,4) }*+{\bullet}="y"
!{(0.5,4.8) }*+{\hphantom{x}^{w_1}}="xim"
!{(2.5,4.3) }*+{\hphantom{x}^{w_0}}="yim"
!{(-1.2,5.2) }*+{\hphantom{x}\star^{\infty}}="i"
!{(5.16,3.3)}*+{}="infint1"
!{(5,3.44)}*+{}="infint2"
!{(6,1.5)}*+{\hphantom{x}\star_{\xi =0}}="inf"
!{(2,1.5)}*+{\hphantom{x}\star_{z_3}}="d"
!{(4,1.5)}*+{\hphantom{x}\star_{z_4}}="h"
!{(1,1.5)}*+{\hphantom{x}\star_{z_2}}="g"
!{(-1,1.5)}*+{\hphantom{x}\star_{z_1}}="e"
"b"-@{-}"f" "a"-@{-}"f" "a"-@{.}"h" "b"-@{.}"e" "b"-@{.}"g" 
"a"-@{.}"d" "i"-@{~}"x" 
"x"-@{-}"f" "y"-@{-}"f" "y"-@{~}"infint1" "infint2"-@{~}"inf"}}
\fbox{
\begin{tikzpicture}[scale=1]
\node [above] at (4,3.83) {${\textbf{(B)}}$};
\node [above] at (2,1.8) {${}_{\spadesuit}$};
\draw (2,2) -- ++(90:1);\draw (2,2) -- ++(162:2);
\draw (2,2) -- ++(18:2);\draw (2,2) -- ++(234:1);
\draw (2,2) -- ++(-54:1);
\draw (1.05,2.3) -- ++(172:1);\draw (1.05,2.3) -- ++(182:1);
\draw (1.05,2.3) -- ++(152:1);\draw (2.95,2.3) -- ++(8:1);
\draw (2.95,2.3) -- ++(-2:1);\draw (2.95,2.3) -- ++(28:1);
\node [above] at (0.1,2.4) {${*}$};
\node [above] at (3.9,2.4) {${*}$};
\node [above] at (2,2.8) {${}_{\clubsuit}$};
\node [above] at (2.6,1) {${}_{\clubsuit}$};
\node [above] at (1.4,1) {${}_{\clubsuit}$};
\draw (2,3) -- ++(80:1);\draw (2,3) -- ++(100:1);
\draw (2,3) -- ++(70:1);\draw (2,3) -- ++(110:1);
\draw (1.4,1.2) -- ++(244:1);\draw (1.4,1.2) -- ++(224:1);
\draw (1.4,1.2) -- ++(254:1);\draw (1.4,1.2) -- ++(214:1);
\draw (2.6,1.2) -- ++(-44:1);\draw (2.6,1.2) -- ++(-34:1);
\draw (2.6,1.2) -- ++(-64:1);\draw (2.6,1.2) -- ++(-74:1);
\node [above] at (2.2,3.83) {${}_{\diamondsuit}$};
\node [above] at (1.8,3.83) {${\circ}$};
\node [above] at (0.65,0.2) {${}_{\diamondsuit}$};
\node [above] at (0.93,0.04) {${\circ}$};
\node [above] at (3.36,0.2) 
{${}_\diamondsuit$};
\node [above] at (3.07,0.04) {${\circ}$};
\end{tikzpicture}
}
\]
\caption{The intersection of the branches 
$\mathcal{S}_L(\mathbf{u})$ and 
$\mathcal{S}_L(\mathbf{v})$ can be visualized
in Fig. (A), according to the discussion in the text. 
We assume the field $L$ contains both $\sqrt{-3}$ and 
$\sqrt{-1}$. In that figure $v_c=B_{0,-1/2}$ is the 
middle point of the edge in $\widehat{\mathcal{T}}_E$ joining 
$w_1=B_{0,-1}$ with $w_{0}=B_{1,0}$.
Fig. (B) shows the $\mathcal{O}_\Omega$ 
representations of the quaternion
group classified according to the fields where 
they are defined. See Table \ref{table1}.
}\label{figura Q8 ghost}
\end{figure}
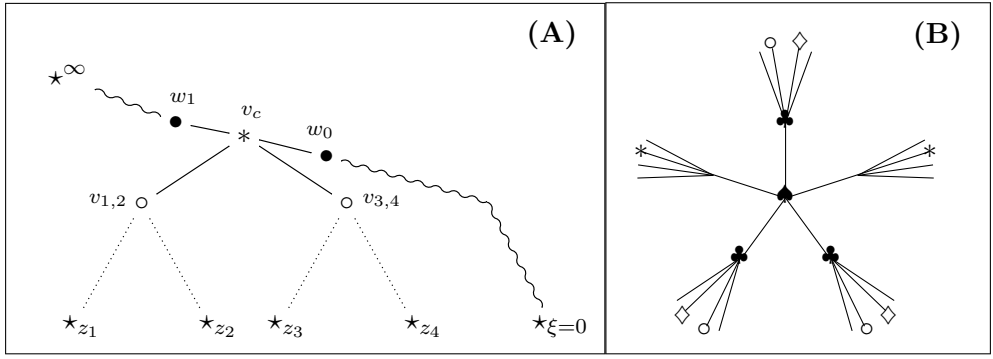

\begin{Lemma}\label{lemma S(u,v) is a ball}
For any splitting field $L$ of $\mathfrak{A}_2\cong
\mathfrak{A}_2''$, we have $v_c\in\widehat{\mathcal{T}}_L$, 
and the branch $\mathcal{S}_L(\mathbf{u}, \mathbf{v})$ is the ball 
with center $v_c$ and radius $\nu(2)/2$.
Furthermore, $v_c$ is a vertex of $\widehat{\mathcal{T}}_L$
if and only if the ramification degree of
$L/\mathbb{Q}_2$ is even.
\end{Lemma}

\begin{proof}
    If $L$ contains both, $\sqrt{-1}$ and $\sqrt{-3}$,
    the result follows from the preceding discussion.
    Otherwise, we define $L'=L(\sqrt{-3})$
    and $L''=L(\sqrt{-1}, \sqrt{-3})$. Note that
    the trivialization $f$ is defined over $L'$,
    and therefore $w_0,w_1,v_c\in\widehat{\mathcal{T}}_{L'}$. 
    Note, however, that $v_c$ is a vertex precisely when
    $L'$ has even ramification degree over $\mathbb{Q}_2$,
    since the radius $-\nu(2)/2$ of $v_c$ as a ball
    of $L''$ is not defined over $\mathbb{Q}_2$.
    Now, the branch
    $\mathcal{S}_{L'}(\mathbf{u}, \mathbf{v})$
    is the intersection of the ball 
    $\mathcal{S}_{L''}(\mathbf{u}, \mathbf{v})$
    with the tree $\widehat{\mathcal{T}}_{L'}$. Since the center $v_c$
    belongs to $\widehat{\mathcal{T}}_{L'}$, and so are the opposite
    leaves $w_0$ and $w_1$, the branch
    $\mathcal{S}_{L'}(\mathbf{u}, \mathbf{v})$ is
    still a ball with the same center and radius.
    This proves the result for $L'$. 
    To prove it for $L$ we observe that
    $L'/L$ is an unramified Galois extension, whence
    $\widehat{\mathcal{T}}_L=
    \widehat{\mathcal{T}}_{L'}^{\mathrm{Gal}(L'/L)}$ as sets. 
    If $L=L'$ there is nothing to prove. Otherwise 
    $\mathrm{Gal}(L'/L)$ is a cyclic group of order $2$ containing 
    a unique non-trivial element $\xi$ of order $2$ that acts via
    $\xi*z=\frac{-1}{2\xi(z)}$, whence it exchanges 
    the vertices $w_0$ and $w_1$. This, in particular, implies
    that $v_c$ is a vertex of the invariant tree 
    $\widehat{\mathcal{T}}_L$,
    as it is the only point of intersection between a
    tree and a path. Now the ball of center $v_c$ and radius
    $\nu(2)/2$ has points in $\widehat{\mathcal{T}}_L$ whose distances
    to $v_c$
    equal the radius $\nu(2)/2$. It follows that these points 
    must be vertices. Note that, in this case, $L'$ must have
    even ramification degree over $\mathbb{Q}_2$, since $v_c$ is 
    a vertex, and the same holds therefore for $L$.
\end{proof}

\begin{Prop}\label{prop int rep in many fields}
    Let $E=\mathbb{Q}_2(\sqrt{-3})$, and let
    $F/\mathbb{Q}_2$ be a ramified quadratic extension.
    Let $F'$ be the other ramified quadratic sub-extension
    of $EF/\mathbb{Q}_2$. Then, the set 
    $\mathrm{IF}_{\rho}^{Q_8}(EF)$ has $6$ elements, 
    $2$ of which are defined over
    $\mathcal{O}_E$, while the other $4$ are defined over 
    either ring $\mathcal{O}_F$ or $\mathcal{O}_{F'}$.
\end{Prop}

\begin{proof}
    Since $e(EF/\mathbb{Q}_2)=2$ and $f(EF/\mathbb{Q}_2)=2$,
    the branch $\mathcal{S}_{EF}(\mathbf{u}, \mathbf{v})$ consists of 
    one vertex and its $5$ neighbors. Since $EF/F$ is unramified,
    the graph $\widehat{\mathcal{T}}_F$ coincides, as a set, with the 
    invariant subgraph 
    $\widehat{\mathcal{T}}_{EF}^{\mathrm{Gal}(EF/F)}$, 
    and the non-trivial element in
    $\mathrm{Gal}(EF/F)$ exchanges $w_0$ and $w_1$ as before.
    The proof for $F'$ is similar.
    Since $v_c$ is not a vertex of $\widehat{\mathcal{T}}_E$, 
    $w_0$ and $w_1$ 
    are the only vertices in $\widehat{\mathcal{T}}_E$ that belong to 
    $\mathcal{S}_{EF}(\mathbf{u}, \mathbf{v})$.
\end{proof}



Let $\Omega=\mathbb{Q}_2(\sqrt{-1},\sqrt{-3},\sqrt{2})$ be the minimal 
field containing the roots of the representatives 
$\lbrace 1, -1,3,-3,2,-2,6,-6 \rbrace$ of 
$\mathbb{Q}_2^*/\mathbb{Q}_2^{2*}$.
Since $e(\Omega/\mathbb{Q}_2)=4$ and $f(\Omega/\mathbb{Q}_2)=2$, 
then the vertices of $S_{\Omega}(\mathbf{u}, \mathbf{v})$ are
the center, its $5$ neighbor, and the $20$ remaining neighbors
of the latter. Now, next result follows from 
Lemma \ref{lemma S(u,v) is a ball} and 
Prop. \ref{prop equiv rep and conj class}: 

\begin{Corollary}\label{Cor int rep of Q8 over Omega}
In the above notations, the set
$\mathrm{IF}_{\rho}^{Q_8}(\Omega)$ has $26$ elements. \qed
\end{Corollary}

\begin{Corollary}\label{Cor int rep of Q8 over Psi}
If $\Psi$ is a totally ramified biquadratic extension of
$\mathbb{Q}_2$, then $\mathrm{IF}_{\rho}^{Q_8}(\Psi)$ has $10$ elements.
\end{Corollary}

\begin{proof}
    Since $\Omega/\Psi$ is unramified, the distance between
    neighbors is the same in both trees, $\widehat{\mathcal{T}}_\Omega$ and
    $\widehat{\mathcal{T}}_\Psi$. The corresponding branches are ball of the 
    same center and radius, by Lemma \ref{lemma S(u,v) is a ball}.
    The only difference is that every vertex has valency
    $5$ in $\widehat{\mathcal{T}}_\Omega$, while they have valency $3$ 
    in $\widehat{\mathcal{T}}_\Psi$.
\end{proof}

\begin{Lemma}\label{lemma perm}
    The action on the tree $\widehat{\mathcal{T}}_\Omega$ defined from 
    the Moebius transformation $\mu(z)=\omega z$ commutes
    with the twisted Galois action defined from the trivialization
    $f$ via Eq. \eqref{eq twisted form}. Furthermore, we have the 
    identities:
    $$\mu \cdot w_0=w_0,\qquad \mu \cdot w_1=w_1,\qquad
    \mu \cdot v_{1,2}=v_{3,4},\qquad \mu \cdot v_{3,4}=w\quad 
    \textnormal{ and }\quad \mu \cdot w=v_{1,2}.$$
\end{Lemma}

\begin{proof}
    It suffices to observe that $\mu$ is induced by the 
    conjugation $\mathbf{x}\mapsto
    \mathbf{m}\mathbf{x}\mathbf{m}^{-1}$,
    where $\mathbf{m}=
    f\left(\frac{\mathtt{j}+\mathtt{1}}2\right)=
    \sbmattrix{-\omega^2}00{-\omega}$, 
    which is an automorphism of the quaternion
    algebra that is defined over $\mathbb{Q}_2$.
    The first two identities follow from the fact that
    $w_0$ and $w_1$ correspond to balls containing $0$ and 
    $\nu(\omega)=0$. Since $\nu(1-\omega)=0$, 
    $\mu$ permutes cyclically the three neighbors of 
    $v_c$ in the tree $\widehat{\mathcal{T}}_L$, for 
    $L=\mathbb{Q}_2(\sqrt{-1}, \sqrt{-3})$, namely 
    $v_{1,2}$, $w$ and $v_{3,4}$. 
    Finally, $\omega z_1-z_3=-\omega^2\sqrt3$, so that 
    $\nu(\omega z_1-z_3)=0$, and therefore 
    $\mu \cdot v_{1,2}=v_{3,4}$. The result follows.
\end{proof}

\begin{Prop}\label{prop rep int fields}
Set $A=\{2,-2,6,-6\}$ and $B=\{3,-1\}$.
If $a,a'\in A$ and $b,b'\in B$, then
$\mathrm{IF}_{\rho}^{Q_8}\big(\Omega/\mathbb{Q}_2(\sqrt{a})\big)=
\mathrm{IF}_{\rho}^{Q_8}\big(\Omega/\mathbb{Q}_2(\sqrt{a'})\big)$
and 
$\mathrm{IF}_{\rho}^{Q_8}\big(\Omega/\mathbb{Q}_2(\sqrt{b})\big)=
\mathrm{IF}_{\rho}^{Q_8}\big(\Omega/\mathbb{Q}_2(\sqrt{b'})\big)$,
while the intersection
$\mathrm{IF}_{\rho}^{Q_8}\big(\Omega/\mathbb{Q}_2(\sqrt{a})\big)\cap
\mathrm{IF}_{\rho}^{Q_8}\big(\Omega/\mathbb{Q}_2(\sqrt{b})\big)$
contains a single class of representations.
\end{Prop}

\begin{proof}
    All these sets contain the representation corresponding to the
    vertex $v_c$. It suffices to check whether the vertices at 
    distance $\frac{\nu(2)}2$,
    in each direction,
    belonging to each tree
    are equal or different. 
    More precisely, 
    let $C_0$, $C_1$, $C_2$, $C_3$ and $C_4$ be
    the connected components of the space obtained by 
    removing $v_c$ from $\widehat{\mathcal{T}}_\Omega$, 
    with $w_0$ in $C_0$
    and $w_1$ in $C_1$. Then the
    action of $\mu$ from Lemma \ref{lemma perm} permutes the 
    three components $C_2$, $C_3$ and $C_4$. Let 
    $E=\mathbb{Q}_2(\sqrt{-3})$, as before.
    Set $F=\mathbb{Q}(\sqrt{x})$ for $x\in A\cup B$.
    It follows from Prop. \ref{prop int rep in many fields} 
    that the branch $\mathcal{S}_{EF}(\mathbf{u}, \mathbf{v})$
    contains a unique vertex in each connected component.
    The trivialization given in Eq. \eqref{eq trivialization}, 
    is defined over the field $EF$. Hence, we can identify $\widehat{\mathcal{T}}_{EF}$ with the untwisted tree
    $\mathcal{T}_{EF}$, consistently for any $F$. With this
    trivialization, any visual limit of one of the connected
    component $C_2$, $C_3$ or $C_4$ corresponds to a uniformizer.
    Let $\pi=\pi_{EF}$ be a uniformizer, and assume it is a visual 
    limit of the component $C_2$. Let $v_x$ be the only neighbor of
    $v_c$ in the intersection $C_2\cap \widehat{\mathcal{T}}_{EF}$.
    This vertex corresponds to the ball $B_{\pi/2}^{[0]}=
    \frac12B_\pi^{[\nu(2)]}$. Since $v_x$ depends only on
    $EF$, it follows that $v_x=v_{3x}$ for $x\in\{2,-2,-1\}$.
    Note that $EF\in\{E(\sqrt{2}),E(\sqrt{-2}),E(\sqrt{-1})\}$, 
    so we can choose a uniformizer for each of these
    fields, and check them for congruence modulo $2$ to see if the
    corresponding vertices coincide or not. We choose
    $\sqrt{2}$ for $E(\sqrt{2})$, $\sqrt{-2}$ for $E(\sqrt{-2})$
    and $1+\sqrt{-1}$ for $E(\sqrt{-1})$.
    
    To make computations in $\Omega$ simpler, we define 
    $\eta=\frac{1+\sqrt{-1}}{\sqrt2}$, an eighth root of unity,
    and set $\eta=1+\epsilon$. Note that $2\mathcal{O}_\Omega=
    \epsilon^4\mathcal{O}_\Omega$. Then we have the following 
    computations:
    $$1+\sqrt{-1}=1+\eta^2=1+(1+\epsilon)^2\equiv
    \epsilon^2\,\, (\text{mod } 2).$$
    $$\sqrt{2}=\frac{1+\sqrt{-1}}\eta\equiv(\epsilon^2)
    (1-\epsilon+\epsilon^2-\epsilon^3)\equiv 
    \epsilon^2+\epsilon^3\,\, (\text{mod } 2).$$
    $$\sqrt{-2}\equiv(1+\epsilon^2)(\epsilon^2+\epsilon^3) \equiv
    \epsilon^2+\epsilon^3\,\, (\text{mod } 2).$$
    All these uniformizers correspond to visual limits
    in the same component, which can be assumed to be $C_2$, since
    they are congruent modulo $\epsilon^3$.
    The congruences modulo 4 tell us that $v_{-2}=v_2\neq v_{-1}$.
    By Prop. \ref{prop int rep in many fields}, the vertices
    corresponding to elements of
    $\mathrm{IF}_{\rho}^{Q_8}(\Omega/F)$
    are precisely the vertices in 
    $\mathcal{S}_{EF}(\mathbf{u}, \mathbf{v})$
    except for $w_0$ and $w_1$, namely
    $v_c$, $v_x$, $\mu\cdot v_x$ and $\mu^2\cdot v_x$. 
    The result follows.
\end{proof}

\begin{Prop}\label{tabla}
The fields of definition for each integral representation
over $\Omega$ can be read in Table \ref{table1}. Those
that correspond to a vertex in Fig. \ref{figura Q8 ghost}(B)
with no symbol are defined only over $\mathcal{O}_\Omega$.
\end{Prop}

\begin{proof}
    The results for quadratic extensions are immediate 
    from Prop. \ref{prop int rep in many fields} and 
    Prop. \ref{prop rep int fields}. Then the results for extension 
    of degree $4$ follows by taking all representations
    defined over a quadratic sub-extension, except for
    the fact that the branch at an extension with 
    ramification index $4$ also contain the vertices marked
    with $\clubsuit$. The result follows.
\end{proof}

\begin{table}
\begin{center}
\begin{tabular}{|c|c||c|c|}
\hline
  $\text{Field $L$}$   &  $ \text{$\sharp \mathrm{IF}_{\rho}^{Q_8}(L)$ }$ &  $\text{Field $L$}$   &  \text{$\sharp \mathrm{IF}_{\rho}^{Q_8}(L)$ } \\
  \hline
$\mathbb{Q}_2(\sqrt{-1})$ & $4$ ($\diamondsuit$, $\spadesuit$) & $\mathbb{Q}_2(\sqrt{-1}, \sqrt{3})$ & $6$
( $*$, $\diamondsuit$, $\spadesuit$)\\
$\mathbb{Q}_2(\sqrt{3})$ & $4$ ($\diamondsuit$, $\spadesuit$) & $\mathbb{Q}_2(\sqrt{-1}, \sqrt{2})$ & $10$
($\circ$, $\diamondsuit$,$\clubsuit$, $\spadesuit$)\\
$\mathbb{Q}_2(\sqrt{-3})$ & $2$ ($*$) & $\mathbb{Q}_2(\sqrt{-1}, \sqrt{6})$ & $10$ 
($\circ$, $\diamondsuit$,$\clubsuit$, $\spadesuit$)\\
$\mathbb{Q}_2(\sqrt{2})$ & $4$ ($\circ$,$\spadesuit$) & $\mathbb{Q}_2(\sqrt{3}, \sqrt{2})$ & $10$ 
($\circ$, $\diamondsuit$,$\clubsuit$, $\spadesuit$)\\
$\mathbb{Q}_2(\sqrt{-2})$ & $4$ ($\circ$,$\spadesuit$) & 
$\mathbb{Q}_2(\sqrt{3}, \sqrt{-2})$ & $10$ 
($\circ$, $\diamondsuit$,$\clubsuit$, $\spadesuit$)\\
$\mathbb{Q}_2(\sqrt{6})$ & $4$ ($\circ$,$\spadesuit$) & $\mathbb{Q}_2(\sqrt{-3}, \sqrt{2})$ & 
$6$( $*$, $\circ$, $\spadesuit$)\\
$\mathbb{Q}_2(\sqrt{-6})$ & $4$ ($\circ$,$\spadesuit$) & 
$\mathbb{Q}_2(\sqrt{-3}, \sqrt{-2})$ & $6$( $*$, $\circ$, $\spadesuit$)\\
\hline
\end{tabular}
\vspace{-0.5cm}
\end{center}
\caption{$\mathcal{O}_{\Omega}$-representations with every possible field 
of definition. The symbols $*$, $\circ$, $\diamondsuit$, $\clubsuit$ 
and $\spadesuit$ refer to the position of the corresponding vertex in 
Fig. \ref{figura Q8 ghost}\textbf{(B)}.}\label{table1}
\end{table}

\section{Global integral representations of \texorpdfstring{$Q_8$}{Q8}}\label{section int rep of ham over global fields}

In this section, we describe the global set 
$\mathrm{IF}_{\rho}^{Q_8}(k)$
of integral
forms of the unique faithful representation
$\rho: Q_8 \to \mathrm{GL}_2(k)$,
when $k$ is a quadratic extension
of $\mathbb{Q}$ with ring of integers $\mathcal{O}_k$.
To do so, we denote 
by $k_{\nu}$ its completion at every place $\nu$, and 
by $\mathcal{O}_{\nu}$ its ring of local integers.
The ad\`ele ring $\mathcal{A}_k$ of $k$ in the subring
of the product $\Pi_{\nu} k_\nu$, over all completions of $k$,
whose coordinates belong to $\mathcal{O}_{\nu}$ at all but a finite number of coordinates. The id\`ele group $J_k:=\mathcal{A}_k^*$
is the group of units of $\mathcal{A}_k$.

By a global order $\mathfrak{D}$, we mean an 
$\mathcal{O}_k$-order. For any 
non-archimedean place $\nu$ we denote by $\mathfrak{D}_\nu$
the completion at $\nu$, which is a $\mathcal{O}_{\nu}$-order. 
It is well known 
that the order $\mathfrak{D}$
is completely determined by the family
$\{\mathfrak{D}_\nu\}_\nu$ in the sense made precise by the
local-global principles \cite[\S 81.E]{omeara}, stated,
for each pair of global orders $(\mathfrak{D}, \mathfrak{D}')$,
as follows:
\begin{itemize}
\item[(I)] $\mathfrak{D}_{\nu}=\mathfrak{D}'_{\nu}$, for every 
$\nu$ outside a finite set of exceptional places,
\item[(II)] if $\mathfrak{D}_{\nu}=\mathfrak{D}'_\nu$ for every 
$\nu$, then $\mathfrak{D}=\mathfrak{D}'$, and
\item[(III)] if $\lbrace \mathfrak{D}(\nu) \rbrace_{\nu}$ is a family 
of local orders such that there is a global order $\mathfrak{D}$ 
satisfying $\mathfrak{D}(\nu)=\mathfrak{D}_{\nu}$, for almost all
place $\nu$, then there is a global order $\mathfrak{D}'$ satisfying 
$\mathfrak{D}(\nu)=\mathfrak{D}'_{\nu}$, for all $\nu$.
\end{itemize}

An order $\mathfrak{D}$
is called a $\nu$-variant of $\mathfrak{D}$ if
$\mathfrak{D}_\mu=\mathfrak{D}'_\mu$ for every $\mu\neq\nu$.
We call it a $\nu$-neighbor if, furthermore,
the completions $\mathfrak{D}_\nu$ and $\mathfrak{D}'_\nu$ 
correspond to neighboring vertices in the tree at $\nu$.

According to \cite[\S1]{ArenasBranches}, there exists a distance map
$\boldsymbol{\rho}$ that associates, to every pair of maximal orders
$\mathfrak{D}$ and $\mathfrak{D}'$, an element in the Galois group
of a suitable extension $\Sigma/k$,
with the property that the orders are isomorphic
precisely when this distance is trivial.
The map is defined by 
$\boldsymbol\rho(\mathfrak{D},\mathfrak{D}')=[N(a),\Sigma/k]$,
where $a=(a_\nu)_\nu\in\mathrm{GL}_2(\mathcal{A}_k)$ 
is an adelic element satisfying
$\mathfrak{D}_\nu=a_\nu\mathfrak{D}_\nu a^{-1}_\nu$
at every place $\nu$, 
$N:\mathrm{GL}_2(\mathcal{A}_k)\rightarrow J_k$
denote the reduced norm on ad\`eles,
and $b\mapsto[b,\Sigma/k]$ is the Artin map on id\`eles.
This is true, in general,
for quaternion algebras splitting at some infinite place, 
but for a matrix algebra the field 
$\Sigma$ is the largest exponent two sub-extension
of the Hilbert class field, as proven in the paragraph
entitled ``Example 1 continued'' at the end of \S2 in \cite{Arenas2003}.
This field is the class field associated to the class group
$k^{*2}J_k^2$ of the id\`ele group $J_k$.
Recall that, for any pair of local maximal orders
corresponding to neighboring vertices there exists
a local matrix interchanging them whose determinant is a 
uniformizer. We conclude that
the distance between $\nu$-neighbors is trivial
if and only if the maximal ideal $P_\nu\subseteq\mathcal{O}_k$
associated to $\nu$ defines a square in the ideal class group.
Let us recall that $h_k(2)$ denotes the 
cardinality of the 
maximal exponent-$2$ subgroup of the class group $\mathbf{Cl}_k$.
Next two propositions prove Th. \ref{main teo 5}:

\begin{Prop}
Let $N \equiv 3 \,\, (\text{mod }8)$ be an positive integer.
Assume $Q_8$ has an integral representation over 
$k=\mathbb{Q}(\sqrt{-N})$, 
i.e., every prime divisor of $N$ is congruent mod $8$ to 
either $1$ or $3$.
Then it has precisely $2h_k(2)$ conjugacy classes of integral 
representations.
\end{Prop}

\begin{proof}
The field $k$ has a unique dyadic place, 
at which $2$ is a uniformizer.
The corresponding completion is 
$k_2\cong\mathbb{Q}_2(\sqrt{-3})$.
Prop. \ref{prop int rep in many fields}
shows that the image of a given representation 
$\rho:Q_8\rightarrow \mathrm{GL}_2(k)$
is contained in precisely two local maximal orders. 
Furthermore, at every non-dyadic place, there is a unique
maximal order containing $\rho(Q_8)$ according to
Cor. \ref{coro rep Q8 over p neq 2} and Prop. 
\ref{prop equiv rep and conj class}. We conclude that,
globally, there are two maximal orders 
$\mathfrak{D}_1$ and $\mathfrak{D}_2$ containing
that image, and they differ only at the dyadic place.
Since the dyadic place is principal, these maximal orders 
are isomorphic. It follows from 
\cite[Th. 7]{ArenasAguiloSaavedra} that each 
corresponds to $h_k(2)$ representations.
\end{proof}

\begin{Prop}
Let $N \equiv a \,\, (\text{mod }8)$, with $a\in\{1,2,5,6,7\}$,
be an positive integer. 
Let $\mathbf{2}$ be the only dyadic place in 
$k=\mathbb{Q}(\sqrt{-N})$, and let $[\mathbf{2}]$ be its image 
in the class group $\mathbf{Cl}_k$.
Assume $Q_8$ has an integral representation over $k$. Then:
\begin{itemize}
\item[(a)] $\sharp \mathrm{IF}_{\rho}^{Q_8}(k)=4h_k(2)$, 
when $[\mathbf{2}]$ is an square, and
\item[(b)] $\sharp \mathrm{IF}_{\rho}^{Q_8}(k)\in \lbrace h_k(2), 3h_k(2)\rbrace $, in any other case.
\end{itemize}
\end{Prop}

\begin{proof}
  The proof of (a) is entirely identical to the preceding one, 
    except for one detail. Now the local maximal orders 
    at the dyadic place containing the image of the 
    representation are $4$. Indeed, they correspond to
    the following type of vertices:
    \begin{itemize}
        \item The central vertex denoted by  $\spadesuit$
        in Fig. \ref{figura Q8 ghost}.
        \item The leaves denoted by either $\diamondsuit$ or
        $\circ$, according to the congruence class of $N$ modulo $8$.
    \end{itemize}
    In case (b), the center $\spadesuit$ corresponds to a maximal order that is not 
    conjugate, over $k$, to the orders corresponding to the 
    leaves, as they are $\mathbf{2}$-neighbors. We conclude
    that vertices of at most one type correspond to
    maximal orders that are conjugate to
    $\mathbb{M}_2(\mathcal{O}_K)$ and yield,
    therefore, integral representations.
\end{proof}

\begin{Example}\label{ex 3hk}
    Let $k=\mathbb{Q}(\sqrt{-5})$. Then a representation of 
    $Q_8$ can be defined using the matrices
    $$\mathbf{i}=\sbmattrix01{-1}0\qquad
    \textnormal{ and }\qquad
    \mathbf{j}=\frac12\sbmattrix1{\sqrt{-5}}{\sqrt{-5}}{-1}.$$
    Then the order $\mathcal{O}_k[\mathbf{i},\mathbf{j}]$ coincides
    with the maximal order $\mathbb{M}_2(\mathcal{O}_k)$ at every 
    place except the dyadic place $\mathbf{2}$, which is not
    principal. Since either of the algebras $k[\mathbf{i}]$
    or $k[\mathbf{j}]$ is isomorphic to the unramified
    quadratic extension over $k$. In this case, the group
    $\langle\mathbf{i},\mathbf{j}\rangle$ is contained in 
    precisely $4$ maximal orders
    located in a ball with a center and four leaves.
    Since $a\mathbf{j}$, for $a\in\mathcal{O}_k$, belongs to $\mathbb{M}_2(\mathcal{O}_k)$
    precisely when $2$ divides $a$, the closest vertex
    to $v_0^{[0]}$ in the branch, which is necessarily a leaf,
    is at distance $\nu(2)$ from that vertex (c.f. Lemma 3.6).
    Since $\nu(2)=2\nu(\pi_{\mathbf{2}})$ for a uniformizer 
    $\pi_{\mathbf{2}}$, we conclude
    that $\mathrm{IF}_{\rho}^{Q_8}(k)$ has
    $3h_k(2)=6$ elements in this case.
\end{Example}

\begin{Example}\label{ex hk}
    Let $k=\mathbb{Q}(\sqrt{-6})$. Then a representation of 
    $Q_8$ can be defined using the matrices
    $$\mathbf{i}=\sbmattrix01{-1}0\qquad
    \textnormal{ and }\qquad
    \mathbf{j}=\sbmattrix{1+\frac12\sqrt{-6}}
    {1-\frac12\sqrt{-6}}{1-\frac12\sqrt{-6}}{-1-\frac12\sqrt{-6}}
    .$$
    Computations are similar to those in the 
    preceding example, except that, in this case 
    $a\mathbf{j}$ belongs to $\mathbb{M}_2(\mathcal{O}_k)$
    precisely when $a$ belongs to the maximal ideal $\mathbf{2}$,
    since $\frac2{\sqrt{-6}}$ is a uniformizer at that place.
    Then the leaf that is closest to $v_0^{[0]}$ is actually a neighbor of this vertex. We conclude
    that $\mathrm{IF}^{Q_8}_{\rho}(k)$ has
    $h_k(2)=2$ elements.
\end{Example}

\section*{Acknowledgements}
The authors would like to thank Professor Jean-Pierre Serre for his valuable comments and suggestions.

\bibliographystyle{amsalpha}
\bibliography{refs.bib}

\end{document}